\documentclass[11pt,reqno]{amsart}
\usepackage[shortlabels]{enumitem}

\usepackage[T1]{fontenc}
\usepackage[utf8]{inputenc}
\usepackage{lmodern}

\usepackage{amsfonts}
\usepackage{amsmath, amsthm, amssymb}

\usepackage{graphicx}
\usepackage{amsthm}
\usepackage{framed}
\usepackage{multicol,multienum}
\usepackage{booktabs} % Allows the use of \toprule, \midrule and \bottomrule in tables
\usepackage{mathrsfs}
\usepackage{tcolorbox}
\usepackage{bm}
\usepackage{subfigure}
\usepackage{epstopdf}
\usepackage{float}   % float package
\usepackage{array}
\usepackage[colorlinks=true,linkcolor=blue,citecolor=blue,urlcolor=red]{hyperref}

\newcommand\R{{\mathbb{R}}}

\newcommand\CC{{\mathbb{C}}}

\newcommand\Z{{\mathbb{Z}}}

\DeclareMathOperator\supp{\text{supp}}
\newtheorem{theorem}{Theorem}[section]
\newtheorem{proposition}[theorem]{Proposition}
\newtheorem{lemma}[theorem]{Lemma}
\newtheorem{corollary}[theorem]{Corollary}

\theoremstyle{definition}
\newtheorem{definition}[theorem]{Definition}

\newcommand{\ip}[2]{\langle #1, #2 \rangle}
\usepackage{mathtools}

\newtheorem{remark}[theorem]{Remark}

\numberwithin{equation}{section}

\begin{document}

\title [Defocusing energy-critical inhomogeneous NLS]{Global well-posedness and scattering for defocusing energy-critical inhomogeneous NLS in dimensions $d\ge 3$}\thanks{$^*$Corresponding authors.}

\author{Bo Yang}
\address{School of Mathematics and Statistics, Hubei Key Laboratory of Engineering Modeling and Scientific Computing, Huazhong University of Science and Technology, Wuhan 430074, Hubei, P.R. China.}
\email{boyang@hust.edu.cn (B. Yang)}

\author{Lei Zhang$^*$}
\address{School of Mathematics and Statistics, Hubei Key Laboratory of Engineering Modeling and Scientific Computing, Huazhong University of Science and Technology, Wuhan 430074, Hubei, P.R. China.}
\email{lei\_zhang@hust.edu.cn (L. Zhang)}

\author{Bin Liu$^*$}
\address{School of Mathematics and Statistics, Hubei Key Laboratory of Engineering Modeling and Scientific Computing, Huazhong University of Science and Technology, Wuhan 430074, Hubei, P.R. China.}
\email{binliu@mail.hust.edu.cn (B. Liu)}

\keywords{Inhomogeneous nonlinear Schr\"odinger equation; Energy-critical; Global well-posedness; Scattering; Concentration-compactness.}

\date{}

\raggedbottom
\setlength{\textfloatsep}{8pt plus 2pt minus 2pt}
\setlength{\intextsep}{8pt plus 2pt minus 2pt}
\setlength{\abovedisplayskip}{8pt plus 2pt minus 3pt}
\setlength{\belowdisplayskip}{8pt plus 2pt minus 3pt}
\setlength{\abovedisplayshortskip}{6pt plus 2pt minus 2pt}
\setlength{\belowdisplayshortskip}{6pt plus 2pt minus 2pt}
\setlength{\floatsep}{8pt plus 2pt minus 2pt}
\allowdisplaybreaks[3]
\newcolumntype{L}[1]{>{\raggedright\arraybackslash}p{#1}}
\emergencystretch=2em

\begin{abstract}
We study the defocusing energy-critical inhomogeneous nonlinear Schr\"odinger equation
\[
  i\partial_tu+\Delta u=|x|^{-b}|u|^{\frac{4-2b}{d-2}}u,
  \qquad (t,x)\in\R\times\R^d,
\]
with initial data $u_0\in\dot H_x^1(\R^d)$, where $d\ge 3$ and $0<b<\min\{2,\frac d2\}$. We prove global well-posedness and scattering for arbitrary non-radial data. The main difficulties are that, when $d\ge 6$, the derivative of the critical nonlinearity is only H\"older continuous, so the short-time perturbation argument cannot be closed in $\dot S^1$, and that the singular coefficient $|x|^{-b}$ breaks translation symmetry. To handle these issues,  we exploit the weak-space structure $|x|^{-b}\in L^{\frac{d}{b},\infty}(\R^d)$, introduce exotic Strichartz norms, and prove a long-time stability theorem for the general energy-critical inhomogeneous nonlinear Schr\"odinger equation. We also show that profiles escaping to spatial infinity are asymptotically linear because of the decay of $|x|^{-b}$. Consequently, almost periodic solutions are compact modulo scaling only, with neither spatial nor frequency center parameters. Combined with the concentration--compactness argument of Kenig--Merle [\emph{Invent. Math.} \textbf{166} (2006), 645--675], this yields the main theorem.
\end{abstract}

\maketitle
%\tableofcontents
%----------------------------------------------------

 \section{Introduction}
% ================================================================

We consider the Cauchy problem for the inhomogeneous nonlinear Schr\"odinger equation (INLS) in $\mathbb{R}^d$ with $d \geq 3$:
\begin{equation}\label{eq:INLS}
\begin{cases}
i\partial_t u + \Delta u = |x|^{-b} |u|^{\alpha} u,\\
u(0, x) = u_0(x) ,
\end{cases}
\end{equation}
under the following assumptions:
\begin{equation}\label{eq:assumptions}
u_0 \in \dot{H}_x^1(\mathbb{R}^d),\quad \alpha=\frac{4-2b}{d-2},\quad 0 < b < \min\left\{2, \frac{d}{2}\right\}.
\end{equation}
This model appears in nonlinear optical wave propagation, where the weight $|x|^{-b}$ describes the spatial inhomogeneity of the medium; see \cite{Gill-2000, Liu-Tripathi-1994}. The singular coefficient also arises as the asymptotic limit of polynomially decaying potentials at spatial infinity; see Genoud--Stuart \cite{Genoud-Stuart-2008}.

Equation \eqref{eq:INLS} is invariant under the scaling symmetry
\begin{equation}\label{eq:scaling}
u(t, x) \mapsto  \lambda^{-\frac{2-b}{\alpha}} u\left(\frac{t}{\lambda^2}, \frac{x}{\lambda}\right), \quad u_{0}(x) \mapsto  \lambda^{-\frac{2-b}{\alpha}} u_0\left(\frac{x}{\lambda}\right),
\end{equation}
which implies that \eqref{eq:INLS} is $\dot{H}^{s_c}_x(\mathbb{R}^d)$-critical with the critical Sobolev index
\begin{equation}\label{eq:critical}
s_c=\frac{d}{2}-\frac{2-b}{\alpha}.
\end{equation}
A standard limiting argument shows that the energy
\begin{equation}\label{eq:energy}
E(u(t)) := \frac12\int_{\mathbb{R}^d} |\nabla u(t,x)|^2 \, dx
+ \frac{1}{\alpha+2}\int_{\mathbb{R}^d} |x|^{-b} |u(t,x)|^{\alpha+2} \, dx
\end{equation}
is conserved.

For $s_c=0$, equivalently $\alpha = \frac{4-2b}{d}$, the scaling \eqref{eq:scaling} preserves the mass:
\begin{equation}\label{eq:mass}
M(u(t)) := \int_{\mathbb{R}^d} | u(t,x)|^2 \, dx .
\end{equation}
This is the mass-critical case. For $d\ge 3$, the corresponding large-data theory in $L_x^2(\R^d)$ was recently established by Liu, Miao, and Zheng \cite{Liu-Miao-Zheng2025}.

For $s_c=1$, equivalently $\alpha = \frac{4-2b}{d-2}$, the scaling preserves the homogeneous Sobolev norm $\|u\|_{\dot H_x^1(\R^d)}:=\|\nabla u\|_{L_x^2(\R^d)}$. Since this is exactly the kinetic part of \eqref{eq:energy}, \eqref{eq:INLS} is called energy-critical.

The main goal of this paper is to establish global well-posedness and scattering for the energy-critical INLS \eqref{eq:INLS} in all dimensions $d\ge 3$. As discussed in Subsection~\ref{relate}, the existing non-radial theory (see, for instance, \cite{Guzman-Murphy-2021,Guzman-Keraani-Xu-2025}) is essentially restricted to dimensions $d\le 5$. In particular, the large-data problem in higher dimensions has remained open. Theorem~\ref{thm:main} resolves this problem.

\subsection{Notations and main results}
We first recall Lorentz spaces, then define solutions to \eqref{eq:INLS}, and finally state the main results.

\subsubsection{Sobolev-Lorentz spaces}
For an interval \(I\subset \mathbb R\), we write \(L_t^pL_x^q(I\times \mathbb R^d)\) for the mixed-norm space of measurable functions \(u:I\times \mathbb R^d\to \mathbb C\) such that
\[
\|u\|_{L_t^pL_x^q(I\times \mathbb R^d)}
:=\left(\int_I\left(\int_{\mathbb R^d}|u(t,x)|^q\,dx\right)^{\frac{p}{q}}dt\right)^{\frac{1}{p}}<\infty,
\]
with the usual modification when \(p=\infty\) or \(q=\infty\). We denote by \(p'\) the H\"older conjugate of \(p\), i.e.
$
\frac1p+\frac1{p'}=1.
$
When \(p=q\), we simply write \(L_{t,x}^p(I\times \mathbb R^d)\).

For a measurable function \(u:\mathbb R^d\to \mathbb C\), define its distribution function and decreasing rearrangement by
\[
d_u(\lambda):=\bigl|\{x\in \mathbb R^d:|u(x)|>\lambda\}\bigr|,\qquad \lambda>0,
\]
and
\[
u^*(s):=\inf\{\lambda>0:d_u(\lambda)\le s\},\qquad s>0.
\]
Here \(|A|\) denotes the Lebesgue measure of a measurable set \(A\subset \mathbb R^d\).

\begin{definition}
Let \(0<p<\infty\) and \(0<q\le \infty\). The Lorentz space \(L^{p,q}(\mathbb R^d)\) consists of all measurable functions \(u:\mathbb R^d\to\mathbb C\) such that
$
\|u\|_{L^{p,q}(\mathbb R^d)}<\infty,
$
where
\[
\|u\|_{L^{p,q}(\mathbb R^d)}
:=
\begin{cases}
\left(\dfrac{q}{p}\displaystyle\int_0^\infty \bigl(s^{\frac{1}{p}}u^*(s)\bigr)^q\,\dfrac{ds}{s}\right)^{\frac{1}{q}},
& 0<q<\infty,\\[1.2ex]
\sup_{s>0}s^{\frac{1}{p}}u^*(s),
& q=\infty.
\end{cases}
\]
\end{definition}

For \(s\in\mathbb R\), define
\[
\langle\nabla\rangle^s u
:=\mathcal F^{-1}\!\bigl((1+|\xi|^2)^{\frac{s}{2}}\mathcal F u\bigr),
\qquad
|\nabla|^s u
:=\mathcal F^{-1}\!\bigl(|\xi|^s\mathcal F u\bigr),
\]
where \(\mathcal F\) and \(\mathcal F^{-1}\) denote the Fourier transform and its inverse, respectively.

\begin{definition}
Let \(s\in\mathbb R\), \(1<p<\infty\), and \(1\le q\le\infty\). We define the inhomogeneous and homogeneous Sobolev--Lorentz spaces by
\[
W^{s;p,q}(\mathbb R^d)
:=\{u\in\mathcal S'(\mathbb R^d):\langle\nabla\rangle^s u\in L^{p,q}(\mathbb R^d)\},
\]
and
\[
\dot W^{s;p,q}(\mathbb R^d)
:=\{u\in\mathcal S'(\mathbb R^d):|\nabla|^s u\in L^{p,q}(\mathbb R^d)\},
\]
equipped with the norms
\[
\|u\|_{W^{s;p,q}(\mathbb R^d)}
:=\|\langle\nabla\rangle^s u\|_{L^{p,q}(\mathbb R^d)},
\qquad
\|u\|_{\dot W^{s;p,q}(\mathbb R^d)}
:=\||\nabla|^s u\|_{L^{p,q}(\mathbb R^d)}.
\]
\end{definition}
\subsubsection{Main results}

We first specify the notion of strong solution to \eqref{eq:INLS}.

\begin{definition}\label{def:intro-strong-solution}
Assume that \eqref{eq:assumptions} holds.
For an interval $I\subset\mathbb{R}$, we define the scattering size space
$
S(I):=L_{t,x}^{\frac{2(d+2)}{d-2}}(I\times\mathbb{R}^d).
$
\begin{enumerate}[label=\upshape(\roman*), leftmargin=*, itemsep=4pt]
\item
A complex-valued function $u:I\times\mathbb{R}^d\to\mathbb{C}$ is called a \emph{strong solution} to \eqref{eq:INLS} if, for every compact interval $J\subset I$,
\[
u\in C_t\dot{H}_x^1(J\times\mathbb{R}^d)\cap S(J),
\]
and the Duhamel formula
\[
u(t)=e^{i(t-t_0)\Delta}u(t_0)-i\int_{t_0}^t e^{i(t-s)\Delta}\bigl(|x|^{-b}|u|^{\alpha}u\bigr)(s)\,ds
\]
holds for all $t,t_0\in I$. The interval $I$ is referred to as the lifespan of $u$.

\item
We say that $u$ \emph{blows up} forward in time if there exists $t_0\in I$ such that
\[
\|u\|_{S([t_0,\sup I))}=\infty.
\]
Blow-up backward in time is defined similarly on $(\inf I,t_0]$.
\end{enumerate}
\end{definition}

\begin{remark}
We take $S(I)$ as the critical spacetime norm, since it is the scaling-critical quantity for the energy-critical problem and the natural norm in the blowup/scattering theory and the concentration--compactness argument.
\end{remark}

\begin{definition}
%[\emph{Admissible pair}]
\label{def:admissible}
A pair $(q,r)$ is called $L^2$-admissible if $\displaystyle \frac{2}{q}=\frac{d}{2}-\frac{d}{r}$ with $2\le q,r\le \infty$ and $(q,r,d)\neq(2,\infty,2)$. In particular, we denote by $\Lambda$ the set of all $L^2$-admissible pairs.
\end{definition}

\begin{definition}
%[\emph{Strichartz norms}]
\label{strichartz-norm}
We define the Strichartz norms
\[
\|u\|_{\dot{S}^0(I)}
:=\sup_{(q,r)\in\Lambda}\|u\|_{L_t^q L_x^{r,2}(I\times\mathbb{R}^d)},
\qquad
\|u\|_{\dot{S}^1(I)}:=\|\nabla u\|_{\dot{S}^0(I)},
\]
and the corresponding dual norms
\[
\|u\|_{\dot{N}^0(I)}
:=\inf_{(q,r)\in\Lambda}\|u\|_{L_t^{q'} L_x^{r',2}(I\times\mathbb{R}^d)},
\qquad
\|u\|_{\dot{N}^1(I)}:=\|\nabla u\|_{\dot{N}^0(I)}.
\]
\end{definition}

Our first result is the local theory in the homogeneous energy space $\dot H_x^1(\R^d)$.

\begin{theorem}[\emph{Local well-posedness in $\dot{H}_x^1$}]
\label{thm:dotH1-local}

For every initial datum \(u_0 \in \dot{H}_x^1(\mathbb{R}^d)\), there exist \(0 < T_{\min}(u_0), T_{\max}(u_0) \le \infty\) and a unique maximal-lifespan strong solution \(u: I(u_0) \times \mathbb{R}^d \to \mathbb{C}\) to \eqref{eq:INLS}, where \(I(u_0) := \bigl(-T_{\min}(u_0), T_{\max}(u_0)\bigr)\).
Moreover, the following statements hold:

\begin{enumerate}[label=\upshape(\roman*), leftmargin=*, itemsep=4pt]
\item For every $L^2$-admissible pair $(q,r)$ and compact interval \(J \subset I(u_0)\), one has
\[
  u\in L_t^q {\dot W}_x^{1;r,2}(J\times \R^d).
\]

\item If \(T_{\max}(u_0) < \infty\), then the solution blows up forward in time; if \(T_{\min}(u_0) < \infty\), then it blows up backward in time.

\item  There exists \(\varepsilon_0 = \varepsilon_0(d,b) > 0\) such that if
\[
\|e^{it\Delta}u_0\|_{S(\mathbb{R})} \le \varepsilon_0,
\]
then \(I(u_0) = \mathbb{R}\), and the solution is global in time and scatters in \(\dot{H}_x^1(\mathbb{R}^d)\). In particular, by Strichartz and Sobolev embedding, this hypothesis is guaranteed when \(\|u_0\|_{\dot{H}_x^1(\mathbb{R}^d)}\) is sufficiently small.
\end{enumerate}

\end{theorem}

The available local theory in $H_x^1(\R^d)$ (cf.\ Aloui--Tayachi \cite{Aloui-Tayachi-2021}) does not cover the present homogeneous setting. We overcome this by proving a long-time stability theorem for the general energy-critical INLS, following Tao--Visan \cite{Tao-Visan-2005} in a Lorentz-space framework. This stability result should also be useful for related focusing and threshold problems.

Our second main result upgrades the local theory to global well-posedness and scattering in both time directions.

\begin{theorem}[\emph{Global well-posedness and scattering in $\dot{H}_x^1$}]
\label{thm:main}

Assume that the conditions in \eqref{eq:assumptions} hold. Then for every initial data \(u_0 \in \dot{H}_x^1(\mathbb{R}^d)\), the Cauchy problem \eqref{eq:INLS} admits a unique global (strong) solution $u \in C_t \dot{H}_x^1(\mathbb{R} \times \mathbb{R}^d)$. Moreover, this solution satisfies
\[
\int_{\mathbb{R}} \int_{\mathbb{R}^d} |u(t,x)|^{\frac{2(d+2)}{d-2}} \, dx \, dt
\le C\bigl(\|u_0\|_{\dot H_x^1}\bigr).
\]
In particular, the solution \(u\) scatters in \(\dot{H}_x^1(\mathbb{R}^d)\); that is, there exist \(u_\pm \in \dot{H}_x^1(\mathbb{R}^d)\) such that
\[
\|u(t) - e^{it\Delta}u_\pm\|_{\dot{H}_x^1(\mathbb{R}^d)} \to 0 \qquad \text{as } t \to \pm\infty.
\]

\end{theorem}

Theorem \ref{thm:main} resolves the defocusing energy-critical scattering problem for the inhomogeneous nonlinear Schr\"odinger equation \eqref{eq:INLS} with large non-radial data in $\dot H_x^1(\R^d)$ throughout the natural range $d\ge 3$ and $0<b<\min\{2,\frac d2\}$. In particular, it fills the high-dimensional gap in the previous non-radial theory and extends earlier low-dimensional or radial results such as \cite{Guzman-Xu-2025,Park-Defocusing-2024}. The proof combines Theorem~\ref{thm:dotH1-local}, the long-time stability result in Proposition~\ref{prop:stability}, and the concentration--compactness argument of Kenig and Merle \cite{Kenig-Merle-2006}; see Subsection~\ref{subsec1.3}.

 \subsection{Related works}\label{relate}
We briefly review previous INLS results and the classical energy-critical NLS theory that guides our argument.

\textsc{The case of NLS.}
For $b=0$, \eqref{eq:INLS} reduces to the standard nonlinear Schr\"odinger equation (NLS),
a basic model for nonlinear pulse propagation and self-focusing in optical media
\cite{Hasegawa-Tappert-1973,Zakharov-Shabat-1972}.
The energy-critical case is given by
\begin{equation}\label{eq:classical-ec-nls}
	\begin{cases}
		i\partial_t v+\Delta v=\mu|v|^{\frac{4}{d-2}}v,\\[2mm]
		v(0)=v_0\in\dot H^1(\R^d),
	\end{cases}
\end{equation}
where $d\ge 3$ and $\mu=\pm1$ with $\mu=1$ corresponds to the defocusing case, while $\mu=-1$ corresponds to the focusing case. Let $W$ denote the ground state solving
$\Delta W+|W|^{\frac{4}{d-2}}W=0.
$. The large-data theory for the classical energy-critical NLS  \eqref{eq:classical-ec-nls} is formulated by
\begin{enumerate}[(i)]
	\item If $\mu=+1$, then every maximal-lifespan solution to \eqref{eq:classical-ec-nls} is global and scatters.
	\item If $\mu=-1$ and
	$
	E(v_0)<E(W)$, $
	\|\nabla v_0\|_{L_x^2}<\|\nabla W\|_{L_x^2},
	$
	then every maximal-lifespan solution to \eqref{eq:classical-ec-nls} is global and scatters.
\end{enumerate}
The local theory in the energy space goes back to Cazenave--Weissler; see also Cazenave's monograph \cite{Cazenave-Weissler-1988,Cazenave-2003}. In the radial defocusing case, Bourgain proved global well-posedness and scattering in three dimensions and introduced the induction-on-energy method together with a localized Morawetz estimate \cite{Bourgain-1999}. Grillakis later gave an alternative proof of global existence in the same setting \cite{Grillakis-2000}. Tao extended the radial argument to all dimensions $d\ge3$, developing the high-dimensional bootstrap needed when the nonlinearity becomes weak \cite{Tao-2005}. For general data, Colliander, Keel, Staffilani, Takaoka, and Tao solved the three-dimensional problem by combining induction on energy with the interaction Morawetz inequality \cite{CKSTT-2008}; Ryckman--Visan treated $d=4$ \cite{Ryckman-Visan-2007}, and Visan completed the non-radial defocusing theory for $d\ge5$ \cite{Visan-2007}. Tao--Visan later isolated the high-dimensional stability theory in the form of exotic Strichartz perturbation estimates \cite{Tao-Visan-2005}.

\begin{table}[H]
	\centering
	\small
	\renewcommand{\arraystretch}{1.15}
	\begin{tabular}{|c|c|c|c|c|}
		\hline
		& $d=3$ & $d=4$ & $d=5$ & $d\ge 6$ \\
		\hline
		radial & Bourgain \cite{Bourgain-1999} & Tao \cite{Tao-2005} & Tao \cite{Tao-2005} & Tao \cite{Tao-2005} \\
		\hline
		nonradial & Colliander et al. \cite{CKSTT-2008} & Ryckman--Visan \cite{Ryckman-Visan-2007} & Visan \cite{Visan-2007} & Visan \cite{Visan-2007} \\
		\hline
	\end{tabular}
	\caption{Defocusing energy-critical NLS.}
	\label{tab:defocusing-nls}
\end{table}

On the focusing side, Kenig--Merle introduced the concentration--compactness/rigidity method in the radial setting \cite{Kenig-Merle-2006}. Their work made the compactness method the standard route for critical dispersive equations. In dimensions $d\ge5$, Killip--Visan removed the radial assumption and proved subthreshold scattering for general data with kinetic energy below that of the ground state \cite{Killip-Visan-2010}. Dodson later settled the non-radial four-dimensional problem \cite{Dodson-2019}. In the focusing case, the non-radial large-data theory in dimension $d=3$ remains open.

\begin{table}[H]
	\centering
	\small
	\renewcommand{\arraystretch}{1.15}
	\begin{tabular}{|c|c|c|c|c|}
		\hline
		& $d=3$ & $d=4$ & $d=5$ & $d\ge 6$ \\
		\hline
		radial & Kenig--Merle \cite{Kenig-Merle-2006} & Kenig--Merle \cite{Kenig-Merle-2006} & Kenig--Merle \cite{Kenig-Merle-2006} & Killip--Visan \cite{Killip-Visan-2010} \\
		\hline
		nonradial & open & Dodson \cite{Dodson-2019} & Killip--Visan \cite{Killip-Visan-2010} & Killip--Visan \cite{Killip-Visan-2010} \\
		\hline
	\end{tabular}
	\caption{Focusing energy-critical NLS.}
	\label{tab:focusing-nls}
\end{table}

\textsc{The case of INLS.} The general Cauchy theory for INLS has been developed in several directions. In the energy space, Farah obtained global existence and blowup criteria \cite{Farah-2016}. Guzm\'an established local and global well-posedness in $H^s$ by classical Strichartz estimates \cite{Guzman-2017}. Kim, Lee, and Seo treated the critical case by weighted-space methods \cite{Kim-Lee-Seo-2021}. A decisive step toward the energy-critical problem was made by Aloui--Tayachi, who introduced a Lorentz-space framework, reached the critical case, and proved continuous dependence and unconditional uniqueness \cite{Aloui-Tayachi-2021}. Very recently, Campos, Correia, and Farah obtained sharp well-posedness and ill-posedness results for general INLS in Sobolev spaces \cite{Campos-Correia-Farah-2025}.

The scattering theory now covers several regimes. In the focusing problem, Farah--Guzm\'an proved radial scattering below the ground state for the three-dimensional cubic model \cite{Farah-Guzman-2017}, and Miao--Murphy--Zheng removed the radial assumption in the same model \cite{Miao-Murphy-Zheng-2021}. Cardoso, Farah, Guzm\'an, and Murphy treated the intercritical non-radial regime \cite{Cardoso-Campos-2022}. In the energy-critical case, Cho--Hong--Lee and Cho--Lee obtained radial three-dimensional results by concentration--compactness and weighted-space arguments \cite{Cho-Lee-2021, Cho-Hong-Lee-2020}. Guzm\'an--Murphy proved below-threshold scattering for the non-radial model case $(d,b)=(3,1)$ \cite{Guzman-Murphy-2021}, and Guzm\'an--Xu extended the non-radial theory to dimensions $3,4,5$ under additional restrictions on the parameter range \cite{Guzman-Xu-2025}. Threshold dynamics have also begun to be studied; see, for example, \cite{Campos-Farah-Murphy-2026, Liu-Yang-Zhang-2024}.

In the defocusing problem, Dinh proved energy scattering in an intercritical range \cite{Dinh-2019}. Aloui--Tayachi later obtained further global existence and scattering results for more general coefficients $K(x)$ with decay at infinity \cite{Aloui-Tayachi-2024}. In the energy-critical regime, Park proved radial global well-posedness and scattering in most of the natural parameter range \cite{Park-Defocusing-2024}, while Guzm\'an--Xu obtained non-radial scattering in dimensions $3,4,5$ under additional restrictions on $b$ \cite{Guzman-Xu-2025}. At mass-critical regularity, Liu--Miao--Zheng recently proved large-data global well-posedness and scattering by carrying out concentration compactness in the Lorentz-space setting \cite{Liu-Miao-Zheng2025}.

\begin{table}[H]
	\centering
	\small
	\renewcommand{\arraystretch}{1.15}
	\begin{tabular}{|c|L{0.45\textwidth}|L{0.25\textwidth}|}
		\hline
		& $d=3,4,5$ & $d\ge 6$ \\
		\hline
		radial & Park \cite{Park-Defocusing-2024} for $0<b< \min\{2,\frac d2\}$ in $d=3$, for $1\le b<2$ in $d=4$, and for $\frac12\le b\le \frac54$ in $d=5$ & Park \cite{Park-Defocusing-2024} for $0<b< \min\{2,\frac d2\}$ \\
		\hline
		nonradial & Guzm\'an--Xu \cite{Guzman-Xu-2025} for $0<b\le \min\{\frac{6-d}{2},\frac4d\}$ & open \\
		\hline
	\end{tabular}
	\caption{Defocusing energy-critical INLS.}
	\label{tab:defocusing-inls-sobolev}
\end{table}

Regarding the entries in Table~\ref{tab:defocusing-inls-sobolev}, from the well-posedness viewpoint, there are two main frameworks: the classical Sobolev setting used by Guzm\'an--Xu \cite{Guzman-Xu-2025}, and the Sobolev--Lorentz framework of Aloui--Tayachi \cite{Aloui-Tayachi-2021}, which is used in Park~\cite{Park-Defocusing-2024} and in the present paper.

\paragraph{\textbf{Our problem.}}
To our knowledge, the defocusing energy-critical problem for general non-radial $\dot H^1$ data in the full range
\[
  d\ge 3,
  \qquad
  0<b<\min\Bigl\{2,\frac d2\Bigr\},
\]
was open. Theorem~\ref{thm:main} resolves it.

\paragraph{\textbf{Comparison with previous results.}}
Our result extends the Lorentz-space local theory of Aloui--Tayachi~\cite{Aloui-Tayachi-2021} and the earlier radial or low-dimensional scattering results of Park, Guzm\'an--Murphy, and Guzm\'an--Xu~\cite{Park-Defocusing-2024,Guzman-Murphy-2021,Guzman-Xu-2025}. In particular, it is the first defocusing scattering result for general non-radial $\dot H^1$ data in all dimensions $d\ge 3$ and $0<b<\min\{2,\frac d2\}$. It closes the high-dimensional gap and removes the additional restrictions on $b$ from previous low-dimensional works.

\subsection{Difficulties and strategy of the proof}\label{subsec1.3}

The proof combines a Sobolev--Lorentz stability theory for the energy-critical INLS with the concentration--compactness/rigidity method of Kenig--Merle.

The first difficulty is perturbative. When $d\ge 6$,
\[
  \alpha=\frac{4-2b}{d-2}<1,
\]
so the derivative of the nonlinearity is only H\"older continuous. As a result, the short-time perturbation argument cannot be closed directly in $\dot S^1$.

To overcome this, we develop an exotic stability theory in Sobolev--Lorentz spaces. We introduce an exotic norm $X(I)$, which controls a scale-invariant fractional derivative of the solution, together with its dual space $Y(I)$ for the Duhamel term. As in Tao--Visan's high-dimensional theory for the classical energy-critical NLS \cite{Tao-Visan-2005}, $S(I)$ is the natural norm for subdividing time intervals, whereas $X(I)$ is the space in which the H\"older nonlinearity becomes perturbative. Interpolating between $S(I)$ and $\dot S^1(I)$ turns $S$-smallness into the $X$-smallness needed for the short-time argument. This leads to the following stability result.

\begin{proposition}[\emph{Long-time stability for the energy-critical INLS}]
\label{prop:stability}

Let \(I \subset \mathbb{R}\) be a compact time interval, and let \(\widetilde{u}\) solve the approximate equation
\begin{equation}\label{eq:approx-eqn-ltp}
(i\partial_t + \Delta)\widetilde{u} = G(\widetilde{u}) + e
\end{equation}
on \(I \times \mathbb{R}^d\) with initial data \(\widetilde{u}(t_0) \in \dot H_x^1(\mathbb R^d) \) , where $G$ is defined in Lemma~\ref{lem:G-definition}. Assume that
\begin{equation}\label{eq:ltp-h1}
\|\widetilde{u}\|_{L_t^\infty \dot{H}_x^1(I \times \mathbb{R}^d)} \lesssim E,
\end{equation}
\begin{equation}\label{eq:ltp-S}
\|\widetilde{u}\|_{S(I)} \lesssim L,
\end{equation}
for some positive constants \(E\) and \(L\). Let \(t_0 \in I\) and assume that \(u(t_0)\) satisfies
\begin{equation}\label{eq:ltp-gap}
\|u(t_0) - \widetilde{u}(t_0)\|_{\dot{H}_x^1(\mathbb{R}^d)} \le E',
\end{equation}
for some positive constant \(E'\). Assume further the smallness conditions
\begin{equation}\label{eq:ltp-free}
\bigl\|e^{i(t-t_0)\Delta}\bigl(u(t_0) - \widetilde{u}(t_0)\bigr)\bigr\|_{S(I)} \le \varepsilon,
\end{equation}
and
\begin{equation}\label{eq:ltp-error}
\|e\|_{\dot{N}^1(I)} \le \varepsilon,
\end{equation}
for some \(0 < \varepsilon < \varepsilon_1(E,E',L)\). Then there exists a unique strong solution \(u: I \times \mathbb{R}^d \to \mathbb{C}\) of \eqref{eq:INLS} with initial data \(u(t_0)\) at time \(t = t_0\), satisfying
\begin{equation}\label{eq:ltp-conc1}
\|u - \widetilde{u}\|_{S(I)} \le C(E,E',L)\varepsilon^c,
\end{equation}
\begin{equation}\label{eq:ltp-conc2}
\|u - \widetilde{u}\|_{\dot{S}^1(I)} \le C(E,L)E',
\end{equation}
and
\begin{equation}\label{eq:ltp-conc3}
\|u\|_{\dot{S}^1(I)} \le C(E,E',L),
\end{equation}
where \(c \in (0,1)\) depends only on \(d\) and \(b\).

\end{proposition}

We write $A\lesssim B$ if $A\le CB$ for some constant $C>0$, and $A\lesssim_M B$ when $C$ depends on $M$. We write $A\sim B$ if both $A\lesssim B$ and $B\lesssim A$, and $A\ll B$ if $A\le cB$ for some sufficiently small $c>0$. For $N\in 2^{\Z}$, define the Littlewood--Paley operators
\[
  \widehat{P_{\le N}u}(\xi):=\varphi(\frac{\xi}{N})\hat u(\xi),
  \qquad
  P_N := P_{\le N}-P_{\le \frac N2},
  \qquad
  P_{>N}:=I-P_{\le N},
\]
 where $\varphi\in C_c^\infty(\R^d)$  be radial with
$
  \varphi(\xi)=1 \text{ for } |\xi|\le 1,
  \varphi(\xi)=0 \text{ for } |\xi|\ge 2.
$
\begin{remark}\label{rem:stability-whole-line}
The same conclusion holds on $I=\R$, $I=[0,\infty)$, or $I=(-\infty,0]$, provided the corresponding hypotheses hold on that interval. One simply partitions the interval into finitely many pieces with small $S$-size and iterates Proposition~\ref{prop:stability}.
\end{remark}

With the perturbative theory in hand, we turn to the Kenig--Merle compactness argument. Following \cite{Kenig-Merle-2006}, we define the critical scattering size in terms of the kinetic energy. For $E_0\ge0$ define
\begin{equation}\label{eq:def-LofE}
  L(E_0)
  :=
  \sup\Bigl\{\|u\|_{ S(I)}:
  \text{$u$ is a strong solution to \eqref{eq:INLS} on $I$ and }
  \sup_{t\in I}\|\nabla u(t)\|_{L_x^2}^2\le E_0\Bigr\}.
\end{equation}
The small-data theory gives $L(E_0)<\infty$ for $E_0$ sufficiently small. If Theorem~\ref{thm:main} were false, there would exist a critical kinetic threshold $E_c\in(0,\infty)$ such that
\begin{equation}\label{eq:def-Ec}
L(E_0)<\infty \quad \text{for } E_0<E_c,
\qquad
L(E_0)=\infty \quad \text{for } E_0>E_c.
\end{equation}
The next step is to extract a minimal kinetic-energy blowup solution. Although $|x|^{-b}$ breaks translation symmetry, it also yields an important simplification, inspired by Miao--Murphy--Zheng~\cite{Miao-Murphy-Zheng-2021}. The linear profile decomposition still produces spatial cores $x_n^j$. If $\frac{|x_n^j|}{\lambda_n^j}\to\infty$, then after rescaling the coefficient becomes $|y+y_n^j|^{-b}$ with $|y_n^j|\to\infty$, so on every fixed compact set in the normalized variables the nonlinearity is negligible and the profile is asymptotically linear. Section~3 makes this precise by constructing global scattering solutions for such far-away profiles. Hence these profiles do not enter the minimal blowup dynamics, and scaling is the only surviving noncompact symmetry. This leads to the following compactness notion.

\begin{definition}[Almost periodicity modulo scaling]\label{def:almost-periodic-scaling}
A maximal-lifespan solution $u:I\times\R^d\to\CC$ with uniformly bounded kinetic energy is said to be \emph{almost periodic modulo scaling} if there exist a frequency scale function $N:I\to(0,\infty)$ and a compactness modulus $C:(0,\infty)\to(0,\infty)$ such that for every $\eta>0$ and every $t\in I$,
\begin{equation}\label{eq:apms-physical}
  \int_{|x|\ge \frac{C(\eta)}{N(t)}} |\nabla u(t,x)|^2\,dx
  +
  \int_{|\xi|\ge C(\eta)N(t)} |\xi|^2|\widehat u(t,\xi)|^2\,d\xi
 \le \eta,
\end{equation}
Equivalently, the orbit
\begin{equation}\label{eq:apms-orbit}
  \Bigl\{N(t)^{-\frac{d-2}{2}}u\bigl(t,N(t)^{-1}x\bigr):t\in I\Bigr\}
\end{equation}
is precompact in $\dot H^1(\R^d)$.
\end{definition}

The concentration--compactness argument then yields a minimal blowup solution which is almost periodic modulo scaling.

\begin{theorem}[Existence of a minimal kinetic-energy blowup solution]\label{prop:critical-element}
Assume that Theorem~\ref{thm:main} fails. Then there exist a maximal-lifespan solution $u_c:I_c\times\R^d\to\CC$ and a number $E_c\in(0,\infty)$ such that
\begin{equation}\label{eq:critical-element-kinetic}
  \sup_{t\in I_c}\|\nabla u_c(t)\|_{L_x^2}^2 = E_c,
\end{equation}
\begin{equation}\label{eq:critical-element-blowup-both-sides}
  \|u_c\|_{ S ([t,\sup I_c))}
  =
  \|u_c\|_{ S ((\inf I_c,t])}
  = \infty
  \qquad \text{for every } t\in I_c,
\end{equation}
and $u_c$ is almost periodic modulo scaling in the sense of Definition~\ref{def:almost-periodic-scaling}.
\end{theorem}

To finish the argument one still has to classify the possible behaviors of this minimal solution. For the classical non-radial energy-critical NLS, one has to distinguish three enemies after normalization: finite-time blowup, a soliton-like solution, and a low-to-high frequency cascade.

For the INLS, the decay of \(|x|^{-b}\) at spatial infinity makes profiles escaping to infinity asymptotically linear, so no spatial center parameter survives. Therefore, all global noncompact behaviors are absorbed into a single global compact scenario, besides the finite-time blowup scenario.

\begin{theorem}[Rigidity scenarios]\label{prop:normalized-rigidity-scenarios}
Assume that Theorem~\ref{thm:main} fails and let $u_c$ be the minimal kinetic-energy blowup solution from Theorem~\ref{prop:critical-element}. Then, after replacing $u_c$ by another minimal kinetic-energy blowup solution if necessary, and after a single time translation and a single scaling, one may assume that exactly one of the following holds:
\begin{enumerate}[(i)]
  \item Finite-time blowup scenario: either $\sup I_c< \infty$ or $|\inf I_c|< \infty$.
  \item Global compact scenario: $I_c=\R$ and
  \begin{equation}\label{eq:normalized-global-lower-bound}
    \inf_{t\in\R} N(t)\ge 1.
  \end{equation}
\end{enumerate}
\end{theorem}

The remainder of the paper is devoted to excluding these two scenarios. Finite-time blowup is ruled out by a reduced Duhamel formula, while the global compact scenario is excluded by a localized virial argument centered at the origin. This completes the proof of Theorem~\ref{thm:main}.

\subsection{Organization of the paper}

The rest of the paper is organized as follows. Section~2 develops the notation, local theory, and stability estimates in the Sobolev--Lorentz framework. Section~3 carries out the concentration--compactness reduction, including the linear profile decomposition, the treatment of far-away profiles, the Palais--Smale condition, and the proof of Theorem~\ref{prop:critical-element}. Section~4 contains the rigidity argument and excludes the two remaining scenarios, thereby proving Theorem~\ref{thm:main}. Appendix~\ref{sec:appendix-lorentz} collects several Sobolev--Lorentz estimates used in the main text.

\section{Perturbation theory and local well-posedness}\label{sec2}

% ================================================================

In this section, we shall establish the stability  for the energy-critical INLS, i.e., Proposition \ref{prop:stability}, and then give the proof of Theorem \ref{thm:dotH1-local}.

Let us consider the general energy-critical INLS equation:
\begin{equation}\label{eq:general-ec-inls}
\begin{cases}
  i\partial_t u+\Delta u=\mu|x|^{-b}|u|^{\alpha}u,\\[2mm]
  u(0,x)=u_0(x),
\end{cases}
\end{equation}
where $\mu = \pm 1$ denotes the defocusing ($\mu=1$) and focusing ($\mu=-1$) cases, respectively, and the parameters $b$ and $\alpha$ satisfy the hypotheses stated in \eqref{eq:assumptions}.

\begin{lemma}\label{lem:G-definition}
Let $G(u):=|x|^{-b}F(u)$, where $F(u):=\mu|u|^{\alpha}u$. Then the following nonlinear estimates hold:
	\begin{equation}\label{eq:basic-pointwise-difference}
		|G(u)-G(v)|\lesssim |x|^{-b}\bigl(|u|^\alpha+|v|^\alpha\bigr)|u-v|,
	\end{equation}
	and\begin{equation}\label{D-basic-nonlinear-estimates}
		\bigl|\nabla\bigl(G(u)-G(v)\bigr)\bigr|
		\lesssim
		|x|^{-b-1}\bigl(|u|^\alpha+|v|^\alpha\bigr)|u-v|
		+
		|x|^{-b}|u|^\alpha|\nabla(u-v)|
		+ E,
	\end{equation}
	where
	\[
	E \lesssim
	\begin{cases}
		|x|^{-b}\bigl(|u|^{\alpha-1}+|v|^{\alpha-1}\bigr)|\nabla v|\,|u-v|,
		& \text{if } \alpha \ge 1, \\[6pt]
		|x|^{-b}|\nabla v|\,|u-v|^\alpha,
		& \text{if } 0<\alpha< 1.
	\end{cases}
	\]
\end{lemma}

\begin{proof}
The inequalities follow by direct computation, so we omit the details.
\end{proof}

For the high-dimensional case $d\ge 6$ one has $\alpha<1$, so the derivative of the energy-critical nonlinearity is only H\"older continuous. Following the Tao--Visan exotic perturbation method \cite{Tao-Visan-2005}, we therefore work with a fractional derivative of order $s<\frac{\alpha}{1+\alpha}$.

Let
\begin{equation}\label{eq:2.8}
0 < s < \frac{\alpha}{1+\alpha}, \quad q_* := \frac{\alpha+2}{s}, \quad \widetilde{q}_* := \frac{\alpha+2}{(1+\alpha)s},
\end{equation}
and we introduce the exotic norms
\begin{align}
\|u\|_{X^0(I)} &:= \|u\|_{L_t^{q_*} L_x^{r_0,2}(I\times\mathbb{R}^d)},\!
&&\frac{2}{q_*} + \frac{d}{r_0} = \frac{d}{2} - 1, \notag \\
\|u\|_{X(I)} &:= \bigl\||\nabla|^s u\bigr\|_{L_t^{q_*} L_x^{r_*,2}(I\times\mathbb{R}^d)},\!
&&\frac{2}{q_*} + \frac{d}{r_*} = \frac{d}{2} - (1-s),\notag \\
\|u\|_{Y(I)} &:= \bigl\||\nabla|^s u\bigr\|_{L_t^{\widetilde{q}_*} L_x^{\widetilde{r}_*,2}(I\times\mathbb{R}^d)},\!
&&\frac{2}{\widetilde{q}_*} + \frac{d}{\widetilde{r}_*} = \frac{d}{2} + (1+s), \notag
\end{align}
where $X^0(I)$ remains tied to one full derivative, $X(I)$ captures the fractional derivative required by the short-time argument, and $Y(I)$ matches the Duhamel term at the critical scaling.

The next few lemmas provide several basic estimates for the nonlinearity, which will play a key role in the perturbative analysis.
\begin{lemma}[Exotic Strichartz estimate]\label{eq:exotic-Strichartz-estimate}
Let $I$ be a compact time interval containing $t_0$, then
\begin{equation}\label{eq:2.12}
\left\|\int_{t_0}^{t} e^{i(t-\tau)\Delta}f(\tau)\,d\tau\right\|_{X(I)}
\lesssim
\|f\|_{Y(I)}.
\end{equation}
\end{lemma}

\begin{proof}
Set $
g(\tau):=|\nabla|^s f(\tau).$,  it suffices to prove
\begin{equation}\label{eq:2.13}
\left\|\int_{t_0}^{t} e^{i(t-\tau)\Delta}g(\tau)\,d\tau\right\|_{L_t^{q_*}L_x^{r_*,2}(I\times\mathbb{R}^d)}
\lesssim
\|g\|_{L_t^{\widetilde q_*}L_x^{\widetilde r_*,2}(I\times\mathbb{R}^d)}.
\end{equation}
Applying the dispersive estimate (cf. Proposition~\ref{prop:dispersive-estimates}) and Minkowski inequality,
\[
\left\|\int_{t_0}^{t} e^{i(t-\tau)\Delta}g(\tau)\,d\tau\right\|_{L_x^{r_*,2}(\mathbb R^d)}
\lesssim
\int_{t_0}^{t} |t-\tau|^{-\theta}\|g(\tau)\|_{L_x^{\widetilde r_*,2}(\mathbb R^d)}\,d\tau,
\]
where \(\theta=d(\frac12-\frac1{r_*})=1-s+\frac{2}{q_*}\).

Finally, by the Hardy–Littlewood–Sobolev inequality, we obtain
\[
\left\|\int_{t_0}^{t} |t-\tau|^{-\theta}\|g(\tau)\|_{L_x^{\widetilde r_*,2}}\,d\tau\right\|_{L_t^{q_*}(I)}
\lesssim
\|g\|_{L_t^{\widetilde q_*}L_x^{\widetilde r_*,2}(I\times\mathbb{R}^d)}.
\]
Hence \eqref{eq:2.13} follows, and therefore \eqref{eq:2.12} is proved.
\end{proof}

\begin{lemma}\label{lem:interpolation}
Let $I$ be a compact time interval and $0\le s\le 1$. Then
\begin{gather}
\|u\|_{\dot N^1(I)}\lesssim \|u\|_{Y(I)},
\label{eq:interp-0}\\
\|u\|_{X^0(I)}\lesssim \|u\|_{X(I)}\lesssim \|u\|_{\dot S^1(I)},
\label{eq:interp-1}\\
\|u\|_{X(I)}
\lesssim
\|u\|_{S(I)}^{\theta_s}\|u\|_{\dot S^1(I)}^{1-\theta_s},
\label{eq:interp-2}\\
\|u\|_{S(I)}
\lesssim
\|u\|_{X(I)}^{c_s}\|u\|_{\dot S^1(I)}^{1-c_s}.
\label{eq:interp-3}
\end{gather}
\end{lemma}

\begin{proof}
The estimate~\eqref{eq:interp-0} and~\eqref{eq:interp-1} follow from Sobolev embedding.
Choose an $L^2$-admissible pair $(q,r)$ so that
\[
\frac1{r_*}=\frac1r-\frac{1-s}{d}.
\]
By H\"older's inequality in time and Sobolev--Lorentz embedding in space,
\[
\bigl\||\nabla|^s u\bigr\|_{L_t^{q_*}L_x^{r_*,2}(\mathbb R^d)}
\lesssim
\|\nabla u\|_{L_t^qL_x^{r,2}(I\times\mathbb R^d)}
\lesssim
\|u\|_{\dot S^1(I)}.
\]
This proves \eqref{eq:interp-1}.

For \eqref{eq:interp-2}, interpolate between $X^0(I)$ and $\dot S^1(I)$ to obtain
\[
\|u\|_{X(I)}
\lesssim
\|u\|_{X^0(I)}^{1-\theta_0}
\|u\|_{\dot S^1(I)}^{\theta_0}
\]
for some $\theta_0\in(0,1)$. On the other hand, by H\"older's inequality in time and Sobolev--Lorentz embedding in space,
\[
\|u\|_{X^0(I)}
\lesssim
\|u\|_{S(I)}^{\theta_1}
\|u\|_{\dot S^1(I)}^{1-\theta_1}
\]
for some $\theta_1\in(0,1)$. Combining these two inequalities gives
\[
\|u\|_{X(I)}
\lesssim
\|u\|_{S(I)}^{(1-\theta_0)\theta_1}
\|u\|_{\dot S^1(I)}^{1-(1-\theta_0)\theta_1},
\]
which is \eqref{eq:interp-2} with
\(
\theta_s:=(1-\theta_0)\theta_1\in(0,1).
\)

By a similar argument, we obtain
\[
\|u\|_{S(I)}
\lesssim
\|u\|_{X^0(I)}^{c_s}
\|u\|_{\dot S^1(I)}^{1-c_s}
\]
for some $c_s\in(0,1)$. Using the first estimate in \eqref{eq:interp-1}, we conclude that
\[
\|u\|_{S(I)}
\lesssim
\|u\|_{X(I)}^{c_s}
\|u\|_{\dot S^1(I)}^{1-c_s},
\]
which is \eqref{eq:interp-3}.
\end{proof}

\begin{lemma}[\emph{Nonlinear estimate $\mathrm{I}$}]\label{lem:nonlinear-estimateI}
For any compact time interval $I$,
\begin{equation}\label{eq:2-16}
\|G(u)\|_{\dot N^1(I)}
\lesssim
\|u\|_{S(I)}^\alpha \|u\|_{\dot S^1(I)},
\end{equation}
and
\begin{equation}\label{eq:2-17}
\|G(u)\|_{Y(I)}
\lesssim
\|u\|_{X^0(I)}^\alpha \|u\|_{X(I)}.
\end{equation}
\end{lemma}

\begin{proof}
We first prove \eqref{eq:2-16}. Since
\[
\nabla G(u)=\nabla(|x|^{-b}F(u))=|x|^{-b}\nabla(F(u))+|x|^{-b-1}F(u),
\]
it suffices to estimate these two terms separately.

For the first term, set
\[
\gamma:=\frac{2(d+2)}{d-2+2b},
\qquad
\rho:=\frac{2d(d+2)}{d^2+4-4b}.
\]
A direct computation shows that $(\gamma,\rho)$ is an $L^2$-admissible pair, and by H\"older inequality,
\begin{align*}
\bigl\||x|^{-b}\nabla(F(u))\bigr\|_{L_t^2L_x^{\frac{2d}{d+2},2}(I\times\mathbb R^d)}
&\lesssim
\|u\|_{S(I)}^\alpha
\|\nabla u\|_{L_t^{\gamma}L_x^{\rho,2}(I\times\mathbb R^d)}.
\end{align*}
For the second term, write
\[
|x|^{-b-1}F(u)=|x|^{-b}|u|^\alpha\,|x|^{-1}u.
\]
Similarly, using H\"older's inequality as above and  Hardy's inequality, we prove \eqref{eq:2-16}.

We now prove \eqref{eq:2-17}. By Lemma~\ref{lem:weighted-fractional-multiplier},
\[
\|G(u)\|_{Y(I)}
\lesssim
\bigl\||\nabla|^sF(u)\bigr\|_{L_t^{\widetilde q_*}L_x^{r,2}(I\times\mathbb R^d)},
\]
where
\(
\frac1r=\frac1{\widetilde r_*}-\frac{b}{d}.
\)
Moreover,
\begin{equation}\label{q*r*}
\frac1{\widetilde q_*}=\frac\alpha{q_*}+\frac1{q_*},
\qquad
\frac1r=\frac\alpha{r_0}+\frac1{r_*}.
\end{equation}
Thus, using Lemma~\ref{lem:fractional-chain} together with \eqref{q*r*}, we deduce
\[
\|G(u)\|_{Y(I)}
\lesssim
\|F'(u)\|_{L_t^{\frac{q_*}{\alpha}}L_x^{\frac{r_0}{\alpha},\infty}(I\times\mathbb R^d)}
\bigl\||\nabla|^s u\bigr\|_{L_t^{q_*}L_x^{r_*,2}(I\times\mathbb R^d)}
\lesssim
\|u\|_{X^0(I)}^\alpha\|u\|_{X(I)},
\]
 This proves \eqref{eq:2-17}.
\end{proof}

\begin{lemma}[\emph{Nonlinear estimate II}]
\label{lem:nonlinear-estimateII}
Let $\sigma := 1-s$. For any compact time interval $I$,
\begin{equation}
\label{eq:coeff-estimate-z-general}
\begin{aligned}
&\bigl\||x|^{-b}F_z(u+v)w\bigr\|_{Y(I)}+\bigl\||x|^{-b}F_{\bar z}(u+v)\bar w\bigr\|_{Y(I)}
\\
&\quad \lesssim
\Bigl(\|u\|_{X(I)}^{\alpha-\frac{s}{\sigma}}\|u\|_{\dot S^1(I)}^{\frac{s}{\sigma}}
+
\|v\|_{X(I)}^{\alpha-\frac{s}{\sigma}}\|v\|_{\dot S^1(I)}^{\frac{s}{\sigma}}\Bigr)
\|w\|_{X(I)}.
\end{aligned}
\end{equation}
\end{lemma}

Note that the restriction $0 < s < \frac{\alpha}{1+\alpha}$ in \eqref{eq:2.8} is necessary here, as it ensures that
\[
\alpha - \frac{s}{\sigma} > 0,
\]
which guarantees the coefficient in the nonlinear estimate carries a positive power of $\|u\|_{X(I)}$.
This positive power is exactly what allows us to close the bootstrap argument within the short-time perturbation framework (cf.\ Proposition \ref{thm:short-pert-sec4}).

\begin{proof}[\textbf{Proof of Lemma \ref{lem:nonlinear-estimateII}}]
By virtue of Lemma~\ref{lem:weighted-fractional-multiplier}, we obtain
\begin{align}
\bigl\||x|^{-b}F_z(u+v)\,w\bigr\|_{Y(I)}
&=
\bigl\||\nabla|^s\bigl(|x|^{-b}(F_z(u+v)w)\bigr)\bigr\|_{L_t^{\widetilde q_*}L_x^{\widetilde r_*,2}(I\times\mathbb R^d)}\notag\\
&\lesssim
\bigl\||\nabla|^s(F_z(u+v)w)\bigr\|_{L_t^{\widetilde q_*}L_x^{r,2}(I\times\mathbb R^d)},\notag
\end{align}
where
$
\frac1{r}:=\frac1{\widetilde r_*}-\frac{b}{d}.
$
Applying Lemma~\ref{lem:fractional-leibniz} with \eqref{q*r*}, we obtain
\begin{align}
\bigl\||\nabla|^s(F_z(u+v)w)\bigr\|_{L_t^{\widetilde q_*}L_x^{r,2}(I\times\mathbb{R}^d)}
\lesssim{}&
\|F_z(u+v)\|_{L_t^{\frac{q_*}{\alpha}}L_x^{\frac{r_0}{\alpha},\infty}(I\times\mathbb{R}^d)}
\,
\bigl\| |\nabla|^s w \bigr\|_{L_t^{q_*} L^{r_*,2}(I\times\mathbb{R}^d)}
\notag\\
&+
\bigl\||\nabla|^sF_z(u+v)\bigr\|_{L_t^{\frac{q_*}{\alpha}}L_x^{m_\sigma,\infty}(I\times\mathbb{R}^d)}
\,
\|w\|_{L_t^{q_*}L_x^{r_0,2}(I\times\mathbb{R}^d)},\label{eq:m-sigma-general}
\end{align}
where
\begin{equation*}
\frac1{m_\sigma}
:=
\frac{\alpha-\frac{s}{\sigma}}{r_0}
+
\frac{\frac{s}{\sigma}}{r_\sigma},
\qquad
\frac{1}{r_\sigma}
:=
\frac{1}{2}-\frac{1-\sigma}{d}-\frac{2}{dq_*}.
\end{equation*}

We first estimate the zeroth-order term. Since $|F_z(z)|\lesssim |z|^\alpha$,  it follows from Lemma~\ref{lem:interpolation} that
\begin{equation}
\label{eq:first-term-general}
\|F_z(u+v)\|_{L_t^{\frac{q_*}{\alpha}}L_x^{\frac{r_0}{\alpha},\infty}(I\times\mathbb R^d)}\|w\|_{X(I)}
\lesssim
\|u+v\|_{X(I)}^{\alpha}
\|w\|_{X(I)}.
\end{equation}

For the second term, we distinguish two cases.

\textsf{Case 1: $\alpha\ge 1$.}
Since $F_z\in C^1(\CC)$, Lemma~\ref{lem:fractional-chain} and Lemma~\ref{lem:interpolation} yields
\begin{align}
\bigl\||\nabla|^s F_z(u+v)\bigr\|_{L_t^{\frac{q_*}{\alpha}}L_x^{m_\sigma,\infty}(I\times\mathbb{R}^d)}
&\lesssim
\bigl\|F_z'(u+v)\bigr\|_{L_t^{\frac{q_*}{\alpha-1}}L_x^{\frac{r_0}{\alpha-1},\infty}(I\times\mathbb{R}^d)}
\bigl\||\nabla|^s(u+v)\bigr\|_{L_t^{q_*}L_x^{r_*,2}(I\times\mathbb{R}^d)} \notag\\
&\lesssim
\|u+v\|_{X(I)}^{\alpha}.
\label{eq:second-term-general-alpha-ge-one}
\end{align}
where we used the identities
\(
\frac1{m_\sigma}=\frac{\alpha-1}{r_0}+\frac1{r_*}
\)
and
\(\frac1{r_*}=\frac{1-\frac{s}{\sigma}}{r_0}+\frac{\frac{s}{\sigma}}{r_\sigma}.
\)

\textsf{Case 2: $0<\alpha<1$.}
Since $F_z$ is H\"older continuous of order
$\alpha$, Lemma~\ref{lem:frac-chain-holder} and Sobolev embedding
give
\begin{align}
\bigl\||\nabla|^sF_z(u+v)\bigr\|_{L_t^{\frac{q_*}{\alpha}}L_x^{m_\sigma,\infty}(I\times\mathbb R^d)}
&\lesssim
\|u+v\|_{L_t^{q_*}L_x^{r_0,2}(I\times\mathbb R^d)}^{\alpha-\frac{s}{\sigma}}
\,
\bigl\||\nabla|^\sigma(u+v)\bigr\|_{L_t^{q_*}L_x^{r_\sigma,2}(I\times\mathbb R^d)}^{\frac{s}{\sigma}},\notag \\
&\lesssim
\bigl\||\nabla|^\sigma(u+v)\bigr\|_{L_t^{q_*}L_x^{r_\sigma,2}(I\times\mathbb R^d)}^{\alpha}.\label{eq:second-term-general-alpha-lt-one}
\end{align}
where we use $\frac{1}{r_{\sigma}}=\frac{1}{r_0}-\frac{\sigma}{d}$.

In Case~1, \eqref{eq:second-term-general-alpha-ge-one} together with Lemma~\ref{lem:interpolation} yields the desired bound. In Case~2, since \(s<\sigma = 1-s <1\), \eqref{eq:second-term-general-alpha-lt-one} and the Gagliardo--Nirenberg inequality give the same conclusion. Therefore, in both cases,
\begin{equation}
\label{eq:second-term-general}
\bigl\||\nabla|^sF_z(u+v)\bigr\|_{L_t^{\frac{q_*}{\alpha}}L_x^{m_\sigma,\infty}(I\times\mathbb R^d)}
\lesssim
(\|u\|_{X(I)}^{\alpha-\frac{s}{\sigma}}\|u\|_{\dot S^1(I)}^{\frac{s}{\sigma}}
+
\|v\|_{X(I)}^{\alpha-\frac{s}{\sigma}}\|v\|_{\dot S^1(I)}^{\frac{s}{\sigma}}).
\end{equation}
Substituting \eqref{eq:first-term-general} and \eqref{eq:second-term-general} into \eqref{eq:m-sigma-general} and using Lemma~\ref{lem:interpolation}, we obtain the desired estimate for \(|x|^{-b}F_z(u+v)\,w\). The estimate for \(|x|^{-b}F_{\bar z}(u+v)\,\bar w\) is identical, so \eqref{eq:coeff-estimate-z-general} follows.
\end{proof}

\begin{proposition}[Short-time perturbation]\label{thm:short-pert-sec4}
Let $I\subset\mathbb R$ be a compact time interval, and let $\widetilde u$ solve the approximate equation
\begin{equation}\label{eq:approx-eqn-sec4}
(i\partial_t+\Delta)\widetilde u=G(\widetilde u)+e
\end{equation}
on $I\times\mathbb R^d$ with initial data $\widetilde u(t_0) \in \dot H^1(\mathbb R^d) $. Assume
\begin{equation}\label{eq:approx-small2-sec4}
\|\widetilde u\|_{L_t^{\infty}\dot H_x^1(I\times\mathbb{R}^d)}\le E,
\end{equation}
\begin{equation}\label{eq:approx-small-sec4}
\|\widetilde u\|_{X(I)}\le \delta,
\end{equation}
where $\delta=\delta(E)>0$ is sufficiently small. Let $t_0\in I$ and let $u(t_0)\in \dot H^1(\mathbb R^d)$ satisfy
\begin{equation}\label{eq:initial-gap-sec4}
\|u(t_0)-\widetilde u(t_0)\|_{\dot H_x^1(\mathbb{R}^d)}\le E',
\end{equation}
and suppose
\begin{equation}\label{eq:smallness-gap2-sec4}
\|e^{i(t-t_0)\Delta}(u(t_0)-\widetilde u(t_0))\|_{X(I)}\le \varepsilon,
\end{equation}
\begin{equation}\label{eq:smallness-gap-sec4}
\|e\|_{\dot N^1(I)}\le \varepsilon
\end{equation}
for some $0<\varepsilon<\varepsilon_0(E,E')$ sufficiently small. Then there exists a unique solution $u:I\times\mathbb R^d\to\mathbb C$ of~\eqref{eq:INLS} with initial data $u(t_0)$ such that
\begin{equation}\label{eq:short-pert-Xs-sec4}
\|u-\widetilde u\|_{X(I)}\lesssim \varepsilon,
\end{equation}
\begin{equation}\label{eq:short-pert-S1-sec4}
\|u-\widetilde u\|_{\dot S^1(I)}\lesssim E',
\end{equation}
\begin{equation}\label{eq:short-pert-u-sec4}
\|u\|_{\dot S^1(I)}\lesssim E+E',
\end{equation}
and moreover
\begin{equation}\label{eq:short-pert-Ys-sec4}
\|G(u)-G(\widetilde u)\|_{Y(I)}\lesssim \varepsilon,
\end{equation}
\begin{equation}\label{eq:short-pert-N1-sec4}
\|G(u)-G(\widetilde u)\|_{\dot N^1(I)}\lesssim E'.
\end{equation}
\end{proposition}

\begin{proof}
Let $w:=u-\widetilde u$. Then
\begin{equation}\label{eq:w-eqn-sec4}
(i\partial_t+\Delta)w=G(u)-G(\widetilde u)-e,
\qquad
w(t_0)=u(t_0)-\widetilde u(t_0).
\end{equation}

\medskip
\noindent\textsf{Step 1 (Bounds for $\widetilde u$ and $u$).} By the Strichartz estimates for \eqref{eq:approx-eqn-sec4}, Lemma~\ref{lem:interpolation} and Lemma~\ref{lem:nonlinear-estimateI} combined with the smallness conditions
\begin{align}
\|\widetilde u\|_{\dot S^1(I)}
&\lesssim
\|\widetilde u\|_{L_t^{\infty}\dot H_x^1(I\times\mathbb R^d)}
+
\|G(\widetilde u)\|_{\dot N^1(I)}
+
\|e\|_{\dot N^1(I)}  \notag \\
&\lesssim
E+\varepsilon+\delta^{\alpha c_s}\|\widetilde u\|_{\dot S^1(I)}^{1+\alpha(1-c_s)}.
\end{align}
Choosing $\delta=\delta(E)$ sufficiently small, and using the standard continuity argument
\begin{equation}\label{eq:utilde-S1-final-sec4}
\|\widetilde u\|_{\dot S^1(I)}\lesssim E.
\end{equation}

Next, Duhamel's formula together with the exotic Strichartz estimate and \eqref{eq:2-17} yields
\begin{align*}
\|e^{i(t-t_0)\Delta}\widetilde u(t_0)\|_{X(I)}
&\lesssim
\|\widetilde u\|_{X(I)}+\|G(\widetilde u)\|_{Y(I)}+\left\|\int_{t_0}^t e^{i(t-\tau)\Delta}e(\tau)\,d\tau\right\|_{X(I)} \\
&\lesssim
\delta+\|\widetilde u\|_{X^0(I)}^\alpha\|\widetilde u\|_{X(I)}+\|e\|_{\dot N^1(I)} \\
&\lesssim
\delta,
\end{align*}
provided $\varepsilon\ll \delta$. Therefore, by the triangle inequality with \eqref{eq:smallness-gap2-sec4}, we have
\[
\|e^{i(t-t_0)\Delta}u(t_0)\|_{X(I)}\lesssim \delta.
\]
By the same argument, we obtain
\begin{equation}\label{eq:u-xs}
\|u\|_{X(I)}\lesssim \delta.
\end{equation}

\medskip
\noindent\textsf{Step 2 ($\dot S^1$-bound for $w$).}
Strichartz estimates yield
\begin{equation}\label{eq:w-S1-pre-sec4}
\|w\|_{\dot S^1(I)}
\lesssim
E'+\varepsilon+\|G(u)-G(\widetilde u)\|_{\dot N^1(I)}.
\end{equation}
If $\alpha\ge1$, it follows from \eqref{D-basic-nonlinear-estimates} and Lemma~\ref{lem:interpolation} with  H\"older and Hardy inequality,
\[
\|G(u)-G(\widetilde u)\|_{\dot N^1(I)}
\lesssim
\bigl(\|u\|_{S(I)}^\alpha+\|\widetilde u\|_{S(I)}^\alpha\bigr)\|w\|_{\dot S^1(I)}
+
\bigl(\|u\|_{S(I)}^{\alpha-1}+\|\widetilde u\|_{S(I)}^{\alpha-1}\bigr)\|\widetilde u\|_{\dot S^1(I)}\|w\|_{S(I)}.
\]
If $0<\alpha<1$, the same argument gives
\[
\|G(u)-G(\widetilde u)\|_{\dot N^1(I)}
\lesssim
\bigl(\|u\|_{S(I)}^\alpha+\|\widetilde u\|_{S(I)}^\alpha\bigr)\|w\|_{\dot S^1(I)}
+
\bigl(\|u\|_{\dot S^1(I)}+\|\widetilde u\|_{\dot S^1(I)}\bigr)\|w\|_{S(I)}^\alpha.
\]
Using \eqref{eq:u-xs}, \eqref{eq:utilde-S1-final-sec4}, and Lemma~\ref{lem:interpolation}, we have
\[
\|u\|_{S(I)}+\|\widetilde u\|_{S(I)}
\lesssim
\delta^{c_s}E^{1-c_s},
\qquad
\|w\|_{S(I)}
\lesssim
\|w\|_{X(I)}^{c_s}\|w\|_{\dot S^1(I)}^{1-c_s}.
\]
Substituting these estimates into \eqref{eq:w-S1-pre-sec4} with the smallness conditions, one has
\[
\|w\|_{\dot S^1(I)}
\lesssim
E'+\varepsilon
+
\delta^{\alpha c_s}E^{\alpha(1-c_s)}\|w\|_{\dot S^1(I)}
+
C(E){\|w\|^{\beta}_{X(I)}}\|w\|_{\dot S^1(I)}^{1-\beta},
\]
where $\beta:=\alpha c_s$. Using Young's inequality and choosing $\delta$ sufficiently small , we obtain
\begin{equation}\label{eq:w-S1-rough-unified-sec4}
\|w\|_{\dot S^1(I)}
\lesssim
E'+C(E,E')\|w\|_{X(I)}.
\end{equation}

\medskip
\noindent\textsf{Step 3 ($X(I)$-bound for $w$).}
By the exotic Strichartz estimate,
\begin{equation}\label{eq:w-X-pre-sec4}
\|w\|_{X(I)}
\lesssim
\varepsilon+\|G(u)-G(\widetilde u)\|_{Y(I)}.
\end{equation}
Writing $v_\theta:=\widetilde u+\theta w$ and using the fundamental theorem of calculus,
\[
G(u)-G(\widetilde u)
=
\int_0^1 |x|^{-b}\bigl(F_z(v_\theta)w+F_{\bar z}(v_\theta)\overline w\bigr)\,d\theta.
\]
Lemma~\ref{lem:nonlinear-estimateII} therefore yields
\begin{equation}\label{eq:w-Y-rough-sec4}
\|G(u)-G(\widetilde u)\|_{Y(I)}
\lesssim
\|w\|_{X(I)}
\Bigl(
\|\widetilde u\|_{X(I)}^{\alpha-\frac{s}{\sigma}}\|\widetilde u\|_{\dot S^1(I)}^{\frac{s}{\sigma}}
+
\|w\|_{X(I)}^{\alpha-\frac{s}{\sigma}}\|w\|_{\dot S^1(I)}^{\frac{s}{\sigma}}
\Bigr),
\end{equation}
where $\sigma=1-s$. Combining \eqref{eq:w-X-pre-sec4}, \eqref{eq:utilde-S1-final-sec4}, and \eqref{eq:w-S1-rough-unified-sec4}, we get
\[
\|w\|_{X(I)}
\lesssim
\varepsilon
+
\delta^{\alpha-\frac{s}{\sigma}}E^{\frac{s}{\sigma}}\|w\|_{X(I)}
+
\|w\|_{X(I)}^{1+\alpha-\frac{s}{\sigma}}\|w\|_{\dot S^1(I)}^{\frac{s}{\sigma}}.
\]
Since $\alpha-\frac{s}{\sigma}>0$, choosing $\delta$ small and then applying continuity gives
\begin{equation}\label{eq:w-X-bootstrap-sec4}
\|w\|_{X(I)}
\lesssim
\varepsilon
+
\|w\|_{X(I)}^{1+\alpha-\frac{s}{\sigma}}\|w\|_{\dot S^1(I)}^{\frac{s}{\sigma}}.
\end{equation}

\medskip
\noindent\textsf{Step 4 (Conclusion).}
Combining \eqref{eq:w-S1-rough-unified-sec4} with \eqref{eq:w-X-bootstrap-sec4}, we obtain
\[
\|w\|_{X(I)}
\lesssim
\varepsilon
+
\|w\|_{X(I)}^{1+\alpha-\frac{s}{\sigma}}\bigl(E'+C(E,E')\|w\|_{X(I)}\bigr)^{\frac{s}{\sigma}}.
\]
Since $\alpha-\frac{s}{\sigma}>0$, the standard continuity argument now gives
\[
\|w\|_{X(I)}\lesssim \varepsilon,
\]
which proves \eqref{eq:short-pert-Xs-sec4}. Returning to \eqref{eq:w-S1-rough-unified-sec4}, we infer that
\[
\|w\|_{\dot S^1(I)}\lesssim E',
\]
Hence \eqref{eq:short-pert-S1-sec4} and \eqref{eq:short-pert-u-sec4} follow from
\[
\|u\|_{\dot S^1(I)}
\le
\|\widetilde u\|_{\dot S^1(I)}+\|w\|_{\dot S^1(I)}
\lesssim
E+E'.
\]
Finally, \eqref{eq:short-pert-Ys-sec4} is immediate from \eqref{eq:w-Y-rough-sec4}, and \eqref{eq:short-pert-N1-sec4} follows from the estimate in Step~2.
\end{proof}

% At this stage, the roles of these two norms are complementary: the short-time theory is formulated in terms of the exotic norm $X(I)$, whereas the interval subdivision in the long-time theory is naturally carried out with respect to the scattering norm $S(I)$. The interpolation property (cf. Lemma \ref{lem:interpolation}) serves precisely to convert the $S$-smallness on each subinterval into the $X$-smallness required to invoke the short-time result.
%
%
%

\begin{proof}[\textbf{Proof of Proposition \ref{prop:stability}.}]
The proof follows by a standard iteration of the short-time perturbation result. We divide \(I\) into finitely many subintervals \(I=\bigcup_{j=0}^{J-1}I_j\), with \(J=J(E,L)\), so that
\[
\|\widetilde u\|_{S(I_j)}\le \eta(E)
\qquad (0\le j\le J-1).
\]
On each \(I_j\), Strichartz estimates for \eqref{eq:approx-eqn-ltp}, Lemma~\ref{lem:nonlinear-estimateI}, and Lemma~\ref{lem:interpolation} give
\[
\|\widetilde u\|_{\dot S^1(I_j)}\lesssim E,
\qquad
\|\widetilde u\|_{X(I_j)}\lesssim \eta^{\theta_s}E^{1-\theta_s}.
\]
Taking \(\eta\) sufficiently small, Proposition~\ref{thm:short-pert-sec4} applies on each \(I_j\).

Starting from \(I_0\), we apply Proposition~\ref{thm:short-pert-sec4} with initial discrepancy controlled by \eqref{eq:ltp-gap}--\eqref{eq:ltp-error}. The resulting endpoint bounds propagate from one subinterval to the next, so the same argument can be iterated on every \(I_j\). Hence
\[
\|u-\widetilde u\|_{X(I_j)}\lesssim \varepsilon,
\qquad
\|u-\widetilde u\|_{\dot S^1(I_j)}\lesssim E',
\qquad
\|u\|_{\dot S^1(I_j)}\lesssim C(E,E',L)
\]
for all \(j\). Finally, Lemma~\ref{lem:interpolation} yields
\[
\|u-\widetilde u\|_{S(I_j)}
\lesssim
\|u-\widetilde u\|_{X(I_j)}^{c_s}\|u-\widetilde u\|_{\dot S^1(I_j)}^{1-c_s}
\lesssim C(E,E',L)\varepsilon^{c_s}.
\]
Summing over \(j\), we obtain \eqref{eq:ltp-conc1}--\eqref{eq:ltp-conc3}.
\end{proof}

\begin{proof}[\textbf{Proof of Theorem~\ref{thm:dotH1-local}}]\label{rem:dotH1-local-from-stability}
The proof follows a standard approximation argument. Choose $u_{0,n}\in H^1\cap\dot H^1$ such that $u_{0,n}\to u_0$ in $\dot H^1$. By the $H^1$ local theory of Aloui--Tayachi \cite{Aloui-Tayachi-2021}, each $u_{0,n}$ generates a maximal-lifespan solution $u_n\in C_tH_x^1\cap S_{\mathrm{loc}}$. Proposition~\ref{prop:stability} then gives a common short-time interval on which $\{u_n\}$ is Cauchy in $C_t\dot H_x^1\cap S$. Passing to the limit yields a unique strong $\dot H^1$ solution. The blowup criterion follows from the same perturbative argument, while the small-data scattering statement is obtained under the sharper hypothesis $\|e^{it\Delta}u_0\|_{S(\R)}\ll1$; the simpler condition $\|u_0\|_{\dot H^1}\ll1$ is then an immediate consequence of Strichartz and Sobolev embedding.
\end{proof}

% ================================================================
\section{Reduction to almost periodic solutions}\label{sec3}
% ================================================================

In this section we begin the concentration--compactness strategy of Kenig--Merle. The global theorem will be
proved by contradiction. Assuming that Theorem~\ref{thm:main} fails, we first
introduce a critical kinetic threshold, then decompose any near-critical blowup
sequence by the standard linear profile decomposition, and finally use the
stability theory from Section~2 to isolate a single nonlinear profile carrying
all of the critical kinetic energy.

We now pass directly to the linear profile decomposition for bounded sequences in $\dot H^1(\R^d)$. This is the first step in extracting a minimal kinetic-energy blowup solution at the threshold $E_c$.

\subsection{The linear profile decomposition}

The nonlinear equation is not translation invariant, but the linear Schr\"odinger
flow still enjoys the usual scaling and translation symmetries. To keep the later
bookkeeping transparent, we use $g$ for the action on spatial data and $T$ for the
induced action on spacetime functions.

\begin{definition}[Scaling--translation actions]\label{def:scaling-translation-group}
For $x_0\in\R^d$ and $\lambda>0$ define the action on spatial data by
\begin{equation}\label{eq:group-action-space}
  [g_{x_0,\lambda}f](x)
  :=
  \lambda^{-\frac{d-2}{2}} f\!\left(\frac{x-x_0}{\lambda}\right).
\end{equation}
For a spacetime function $u$ define the induced action by
\begin{equation}\label{eq:group-action-time}
  [T_{x_0,\lambda}u](t,x)
  :=
  \lambda^{-\frac{d-2}{2}}u\!\left(\frac{t}{\lambda^2},\frac{x-x_0}{\lambda}\right).
\end{equation}
\end{definition}

For the profile parameters extracted below we abbreviate
\[
  g_n^j := g_{x_n^j,\lambda_n^j},
  \qquad
  T_n^j := T_{x_n^j,\lambda_n^j},
  \qquad
  g_n := g_{x_n,\lambda_n},
  \qquad
  T_n := T_{x_n,\lambda_n}.
\]

\begin{proposition}[Linear profile decomposition]\label{prop:linear-profile-decomposition}
Let $\{u_n\}_{n\ge 1}$ be a bounded sequence in $\dot H^1(\R^d)$. Then, after
passing to a subsequence, there exist an index set
$J^*\in\mathbb N\cup\{\infty\}$, profiles $\phi^j\in\dot H^1(\R^d)$, parameters
\(
  \lambda_n^j>0,
  x_n^j\in\R^d,
  t_n^j\in\R,
\)
and remainders $w_n^J\in\dot H^1(\R^d)$ such that for every finite $J$,
\begin{equation}\label{eq:linear-profile-decomp}
  u_n
  =
  \sum_{j=1}^J g_n^j\bigl[e^{it_n^j\Delta}\phi^j\bigr] + w_n^J.
\end{equation}
Moreover the following hold.
\begin{enumerate}[(i)]

\item \textsf{Asymptotic vanishing of remainder.}
\begin{equation}\label{eq:linear-remainder-small-S}
  \lim_{J\to\infty}\limsup_{n\to\infty}
  \bigl\|e^{it\Delta}w_n^J\bigr\|_{S(\R)} = 0.
\end{equation}

\item \textsf{Kinetic energy decoupling.}
For every finite $J$,
\begin{equation}\label{eq:kinetic-decoupling-linear}
  \|\nabla u_n\|_{L_x^2}^2
  -
  \sum_{j=1}^J \|\nabla \phi^j\|_{L_x^2}^2
  -
  \|\nabla w_n^J\|_{L_x^2}^2
  \longrightarrow 0.
\end{equation}

\item \textsf{Weak decoupling of the remainder in each profile frame.}
For every fixed $j$ and every finite $J\ge j$,
\begin{equation}\label{eq:strong-decoupling-remainder}
  e^{-it_n^j\Delta}(g_n^j)^{-1}w_n^J \rightharpoonup 0,
  \qquad\text{weakly in } \dot H^1(\R^d).
\end{equation}

\item \textsf{Asymptotic orthogonality of parameters.} For $j\neq k$,
\begin{equation}\label{eq:profile-parameter-orthogonality}
  \frac{\lambda_n^j}{\lambda_n^k}
  +
  \frac{\lambda_n^k}{\lambda_n^j}
  +
  \frac{|x_n^j-x_n^k|^2}{\lambda_n^j\lambda_n^k}
  +
  \frac{\bigl|t_n^j(\lambda_n^j)^2-t_n^k(\lambda_n^k)^2\bigr|}{\lambda_n^j\lambda_n^k}
  \longrightarrow \infty.
\end{equation}
\end{enumerate}
After refining the subsequence further, one may also assume that either $t_n^j\equiv 0$ or $t_n^j\to\pm\infty$,
and either $x_n^j\equiv 0$ or $\frac{|x_n^j|}{\lambda_n^j}\to\infty$.
\end{proposition}

\subsection{Nonlinear profiles and Palais--Smale condition}

\begin{proposition}[Profiles far from the origin scatter]\label{prop:far-away-profile}
Let $\phi\in\dot H^1(\R^d)$ and let sequences
$\lambda_n>0$, $ x_n\in\R^d$, $t_n\in\R$, satisfy
\begin{equation}\label{eq:far-away-assumption}
  \frac{|x_n|}{\lambda_n}\to\infty,
  \qquad
  t_n\equiv 0 \ \text{or}\ \ t_n\to\pm\infty.
\end{equation}
Set $g_n:=g_{x_n,\lambda_n}$ and define the initial data
\begin{equation}\label{eq:far-away-initial-data}
  \phi_n(x)
  :=
  g_n\bigl[e^{it_n\Delta}\phi\bigr](x).
\end{equation}
Then for all sufficiently large $n$ there exists a global solution $v_n$ to
\eqref{eq:INLS} with $v_n(0)=\phi_n$ and
\begin{equation}\label{eq:far-away-global-bounds}
  \|v_n\|_{S(\R)} + \|v_n\|_{\dot S^1(\R)} \lesssim_\phi 1.
\end{equation}
Moreover, for every $\varepsilon>0$ there exist $N(\varepsilon)$ and
$\psi\in C_c^{\infty}(\R\times\R^d)$ such that for all $n\ge N(\varepsilon)$,
\begin{equation}\label{eq:far-away-compact-approx}
  \Bigl\|\lambda_n^{\frac{d-2}{2}}v_n\bigl(\lambda_n^2(\,\cdot\,-t_n),\lambda_n\,\cdot\,+x_n\bigr)-\psi\Bigr\|_{S(\R)}
  < \varepsilon.
\end{equation}
\end{proposition}

\begin{proof}
Fix a radial cutoff function $\chi\in C^{\infty}(\R^d)$ satisfying
\[
  0\le \chi\le 1,
  \qquad
  \chi(z)=0 \ \text{for } |z|\le \frac14,
  \qquad
  \chi(z)=1 \ \text{for } |z|\ge \frac12.
\]
For each $n$ define
\begin{equation}\label{eq:chi-direct-definition}
  \chi_n(y) := \chi\!\left(\frac{y+y_n}{|y_n|}\right),
\end{equation}
where $y_n := \frac{x_n}{\lambda_n}$, so that $|y_n|\to\infty$.
Then
\[
  \chi_n(y)=0 \quad \text{if } |y+y_n|\le \frac14|y_n|,
  \qquad
  \chi_n(y)=1 \quad \text{if } |y+y_n|\ge \frac12|y_n|,
\]
and, for every multi-index $\beta$,
\begin{equation}\label{eq:chi-symbol-bounds}
  |\partial^\beta \chi_n(y)| \lesssim_\beta |y_n|^{- |\beta|}.
\end{equation}
We then define the annular cutoff operator
\begin{equation}\label{eq:Pn-choice-far-away}
  P_n := P_{|y_n|^{-\theta}\le \cdot\le |y_n|^{\theta}},
\end{equation}
where $0<\theta<\min\!\left\{1,\frac{b}{2-b}\right\}$.
For $T>0$ let
\begin{equation}\label{eq:far-away-time-window}
  I_{n,T}:=[a^-_{n,T},a^+_{n,T}]
  :=[-\lambda_n^2 t_n-\lambda_n^2T,
      -\lambda_n^2 t_n+\lambda_n^2T].
\end{equation}
Define the approximate solution $\widetilde v_{n,T}$ by
\begin{equation}\label{eq:far-away-approx-central-physical}
  \widetilde v_{n,T}(t)
  :=
  g_n\bigl[\chi_n P_n e^{i(\lambda_n^{-2}t+t_n)\Delta}\phi\bigr],
  \qquad t\in I_{n,T},
\end{equation}
and extend it linearly outside $I_{n,T}$ by
\begin{equation}\label{eq:far-away-approx-outside-physical}
  \widetilde v_{n,T}(t)
  :=
  \begin{cases}
    e^{i(t-a^+_{n,T})\Delta}\widetilde v_{n,T}(a^+_{n,T}), & t>a^+_{n,T},\\[0.6ex]
    e^{i(t-a^-_{n,T})\Delta}\widetilde v_{n,T}(a^-_{n,T}), & t<a^-_{n,T}.
  \end{cases}
\end{equation}
We verify the hypotheses of Proposition~\ref{prop:stability}.

\medskip
\noindent
\textbf{(a) Uniform bounds.}
We use the change of variables
\begin{equation}\label{eq:far-away-change-of-variables}
  t=\lambda_n^2(s-t_n),
  \qquad
  x=\lambda_n y+x_n.
\end{equation}
Then on the central interval $|s|\le T$ we have
\begin{equation}\label{eq:far-away-central-normalized}
  \lambda_n^{\frac{d-2}{2}}\widetilde v_{n,T}(\lambda_n^2(s-t_n),\lambda_n y+x_n)
  =\chi_n(y)\Phi_n(s,y),
  \qquad
  \Phi_n(s,y):=e^{is\Delta}P_n\phi(y).
\end{equation}
Let $2\le r\le \frac{2d}{d-2}$ and $\nabla f\in  L_y^{r,2}(\R^d)$, we use H\"older's inequality and Sobolev embedding
\begin{align*}
  \|\nabla_y(\chi_n f)\|_{L_y^{r,2}}
  &\le \|\chi_n\nabla_y f\|_{L_y^{r,2}} + \|\nabla_y\chi_n\,f\|_{L_y^{r,2}} \\
  &\lesssim \|\nabla_y f\|_{L_y^{r,2}} + \|\nabla_y\chi_n\|_{L_y^d}\|f\|_{L_y^{\frac{dr}{d-r},2}} \\
  &\lesssim \|\nabla_y f\|_{L_y^{r,2}}.
\end{align*}
Taking $r=2$ gives the uniform multiplier bound on $\dot H^1$, and the same argument with $L^2$--admissible $r$ gives the corresponding bound on $\dot S^1$.
Applying these bounds to $f=\Phi_n(s)$ and using homogeneous Strichartz on $[-T,T]$, we obtain
\begin{align}
  \|\chi_n\Phi_n\|_{L_s^{\infty}\dot H_y^1([-T,T])}
  &\lesssim \|\Phi_n\|_{L_s^{\infty}\dot H_y^1([-T,T])}
  \lesssim \|\phi\|_{\dot H_y^1}, \label{eq:far-away-central-H1-bound}\\
  \|\chi_n\Phi_n\|_{\dot S^1([-T,T])}
  &\lesssim \|\Phi_n\|_{\dot S^1([-T,T])}
  \lesssim \|\phi\|_{\dot H_y^1}. \label{eq:far-away-central-S1-bound}
\end{align}
By Sobolev embedding in the spatial variable, \eqref{eq:far-away-central-S1-bound} also gives
\begin{equation}\label{eq:far-away-central-S-bound}
  \|\chi_n\Phi_n\|_{S([-T,T])}\lesssim \|\phi\|_{\dot H_y^1}.
\end{equation}
The outer pieces are free solutions with endpoint data
\[
  f_{n,T}^{\pm}:=\chi_n e^{\pm iT\Delta}P_n\phi.
\]
In the same way, using \eqref{eq:far-away-central-H1-bound} and homogeneous Strichartz on $(a^+_{n,T},\infty)$ and $(-\infty,a^-_{n,T})$ gives
\begin{equation}\label{eq:far-away-approx-bounds}
  \sup_{n,T}
  \Bigl(
    \|\widetilde v_{n,T}\|_{L_t^{\infty}\dot H_x^1(\R)}
    + \|\widetilde v_{n,T}\|_{S(\R)}
    + \|\widetilde v_{n,T}\|_{\dot S^1(\R)}
  \Bigr)
  \lesssim_\phi 1.
\end{equation}

\medskip
\noindent
\textbf{(b) Initial-data matching.}
We claim that
\begin{equation}\label{eq:far-away-initial-gap}
  \lim_{T\to\infty}\limsup_{n\to\infty}
  \|\widetilde v_{n,T}(0)-\phi_n\|_{\dot H_x^1(\R^d)}=0.
\end{equation}
If $t_n\equiv0$, then $0\in I_{n,T}$ and
\[
  \widetilde v_{n,T}(0,x)=g_n[\chi_nP_n\phi](x).
\]
Hence, by the definition of $g_n$,
\[
  \|\widetilde v_{n,T}(0)-\phi_n\|_{\dot H_x^1}
  =\|(\chi_nP_n-1)\phi\|_{\dot H_y^1}\to0,
\]
since $P_n\to I$ strongly on $\dot H_y^1$, $\chi_n\to1$ pointwise, and multiplication by $\chi_n$ is uniformly bounded on $\dot H_y^1$.

If $t_n\to+\infty$, then $0>a_{n,T}^+$ for $n$ large, and by \eqref{eq:far-away-approx-outside-physical},
\[
  \widetilde v_{n,T}(0,x)
  =e^{-ia_{n,T}^+\Delta_x}\widetilde v_{n,T}(a_{n,T}^+,x)
  =g_n e^{it_n\Delta_y}\!\left[e^{-iT\Delta_y}\chi_nP_n e^{iT\Delta_y}\phi\right](x).
\]
Therefore,
\[
  \|\widetilde v_{n,T}(0)-\phi_n\|_{\dot H_x^1}
  =\|(\chi_nP_n-1)e^{iT\Delta_y}\phi\|_{\dot H_y^1}\to0
\]
for each fixed $T$. The case $t_n\to-\infty$ is identical. This proves \eqref{eq:far-away-initial-gap}.

\medskip
\noindent
\textbf{(c) Small perturbation.}
Set
\[
  e_{n,T}:=i\partial_t\widetilde v_{n,T}+\Delta_x\widetilde v_{n,T}-G(\widetilde v_{n,T}).
\]
We claim that
\begin{equation}\label{eq:far-away-error-small}
  \lim_{T\to\infty}\limsup_{n\to\infty}\|e_{n,T}\|_{\dot N^1(\R)}=0.
\end{equation}
We treat $I_{n,T}$ and $\R\setminus I_{n,T}$ separately.

\smallskip
\noindent
\emph{Central interval.}
Using \eqref{eq:far-away-change-of-variables} and \eqref{eq:far-away-central-normalized}, we obtain
\begin{align}
  \lambda_n^{\frac{d+2}{2}}e_{n,T}\bigl(\lambda_n^2(s-t_n),\lambda_n y+x_n\bigr)
  &= \Delta_y\chi_n(y)\,\Phi_n(s,y)
     +2\nabla_y\chi_n(y)\cdot\nabla_y\Phi_n(s,y) \notag\\
  &\qquad
     -|y+y_n|^{-b}\chi_n(y)^{\alpha+1}|\Phi_n(s,y)|^\alpha\Phi_n(s,y).
     \label{eq:central-error-decomp-far-away}
\end{align}
Denote the first two terms by $e_{n,T}^{\mathrm{lin}}$ and the last one by $e_{n,T}^{\mathrm{nl}}$.

For $e_{n,T}^{\mathrm{lin}}$, one derivative in $y$ produces terms of the form
\[
  \partial_y^j\chi_n\,\partial_y^{3-j}\Phi_n,
  \qquad j=1,2,3.
\]
Since $P_n$ is supported in $|\xi|\lesssim |y_n|^\theta$,
\[
  \|\nabla_y^k P_n\phi\|_{L_y^2}\lesssim |y_n|^{(k-1)\theta}\|\phi\|_{\dot H_y^1},
  \qquad k=1,2,3.
\]
Hence, by \eqref{eq:chi-symbol-bounds}, Bernstein inequality, and the dual admissible norm
$L_s^1L_y^{2,2}$,
\begin{equation}\label{eq:linear-error-central-estimate}
  \|\nabla_y e_{n,T}^{\mathrm{lin}}\|_{L_s^1L_y^{2,2}([-T,T]\times\R^d)}
  \lesssim_{T,\phi}
  \sum_{j=1}^3 |y_n|^{-j}|y_n|^{|2-j|\theta}
  \longrightarrow0.
\end{equation}

For $e_{n,T}^{\mathrm{nl}}$, on $\supp\chi_n$ we have $|y+y_n|\ge \frac14|y_n|$, so
\begin{equation}\label{eq:far-away-weight-bounds}
  |y+y_n|^{-b}\lesssim |y_n|^{-b},
  \qquad
  \bigl|\nabla_y\bigl(|y+y_n|^{-b}\chi_n^{\alpha+1}\bigr)\bigr|
  \lesssim |y_n|^{-b-1}.
\end{equation}
Moreover,
\[
  \nabla_y e_{n,T}^{\mathrm{nl}}
  =-\nabla_y\bigl(|y+y_n|^{-b}\chi_n^{\alpha+1}\bigr)|\Phi_n|^\alpha\Phi_n
   -|y+y_n|^{-b}\chi_n^{\alpha+1}\nabla_y\bigl(|\Phi_n|^\alpha\Phi_n\bigr),
\]
and
\[
  |\nabla_y(|\Phi_n|^\alpha\Phi_n)|\lesssim |\Phi_n|^\alpha|\nabla_y\Phi_n|.
\]
The frequency cutoffs imply
\begin{equation}\label{eq:Phi-Linfty-bound-far-away}
\begin{aligned}
&\|\Phi_n\|_{L_s^\infty L_y^\infty([-T,T]\times\R^d)}
  \lesssim |y_n|^{\theta\frac{d-2}{2}}\|\phi\|_{\dot H_y^1},\\
&\|\Phi_n\|_{L_s^\infty L_y^{\frac{2d}{d-2},2}([-T,T]\times\R^d)}
 +\|\nabla_y\Phi_n\|_{L_s^\infty L_y^2([-T,T]\times\R^d)}
 \lesssim \|\phi\|_{\dot H_y^1},
\end{aligned}
\end{equation}
and
\begin{equation}\label{eq:Phi-low-frequency-bound-far-away}
  \|\Phi_n\|_{L_s^\infty L_y^{\frac{2d}{d+2},2}([-T,T]\times\R^d)}
  \lesssim \|\phi\|_{\dot H_y^1},
\end{equation}
while Bernstein inequality yields
\begin{equation}\label{eq:gradPhi-rho-bound-far-away}
  \|\nabla_y\Phi_n\|_{L_s^\infty L_y^{\rho,2}([-T,T]\times\R^d)}
  \lesssim |y_n|^{\theta(b-1)}\|\phi\|_{\dot H_y^1},
  \qquad
  \rho=\frac{2d}{d-2+2b}.
\end{equation}
Therefore, Hölder inequality together with
\eqref{eq:far-away-weight-bounds}--\eqref{eq:gradPhi-rho-bound-far-away} gives
\begin{align}
  \|\nabla_y e_{n,T}^{\mathrm{nl}}\|_{L_s^2L_y^{\frac{2d}{d+2},2}([-T,T]\times\R^d)}
  &\lesssim_T |y_n|^{-b-1}\|\Phi_n\|_{L_s^\infty L_y^\infty}^\alpha
           \|\Phi_n\|_{L_s^\infty L_y^{\frac{2d}{d+2},2}} \notag\\
  &\qquad + |y_n|^{-b}\|\Phi_n\|_{L_s^\infty L_y^{\frac{2d}{d-2},2}}^\alpha
           \|\nabla_y\Phi_n\|_{L_s^2L_y^{\rho,2}} \notag\\
  &\lesssim_{T,\phi}
    |y_n|^{-b-1+\theta(2-b)}+|y_n|^{-b+\theta(b-1)}
  \longrightarrow0.
  \label{eq:nonlinear-error-central-estimate}
\end{align}
Since $L_s^1L_y^{2,2}$ and $L_s^2L_y^{\frac{2d}{d+2},2}$ are admissible dual norms for $\dot N^1$, \eqref{eq:linear-error-central-estimate} and \eqref{eq:nonlinear-error-central-estimate} imply
\begin{equation}\label{eq:central-error-small-far-away}
  \limsup_{n\to\infty}\|e_{n,T}\|_{\dot N^1(I_{n,T})}=0
  \qquad\text{for each fixed }T>0.
\end{equation}

\smallskip
\noindent
\emph{Outer intervals.}
On $\R\setminus I_{n,T}$, the function $\widetilde v_{n,T}$ is an exact free solution, hence
\[
  e_{n,T}=-G(\widetilde v_{n,T}).
\]
We first show
\begin{equation}\label{eq:far-away-linear-tail-small}
  \lim_{T\to\infty}\limsup_{n\to\infty}
  \|\widetilde v_{n,T}\|_{S(\R\setminus I_{n,T})}=0.
\end{equation}
It suffices to consider the right-hand tail. For $s>T$,
\[
  \lambda_n^{\frac{d-2}{2}}\widetilde v_{n,T}\bigl(\lambda_n^2(s-t_n),\lambda_n y+x_n\bigr)
  =e^{i(s-T)\Delta_y}\!\left[\chi_n e^{iT\Delta_y}P_n\phi\right](y).
\]
Hence
\begin{align*}
  &\Bigl\|\lambda_n^{\frac{d-2}{2}}\widetilde v_{n,T}\bigl(\lambda_n^2(s-t_n),\lambda_n y+x_n\bigr)-e^{is\Delta_y}\phi\Bigr\|_{S((T,\infty))}\\
  &\qquad\le
  \|e^{i(s-T)\Delta_y}(\chi_n e^{iT\Delta_y}P_n\phi-e^{iT\Delta_y}\phi)\|_{S((T,\infty))}
  +\|e^{is\Delta_y}(P_n\phi-\phi)\|_{S((T,\infty))}.
\end{align*}
For each fixed $T$, both terms tend to $0$ as $n\to\infty$; the first by the same argument as in part~(b), and the second because $P_n\phi\to\phi$ in $\dot H_y^1$. Since $e^{is\Delta_y}\phi$ has vanishing $S$-tail as $T\to\infty$, \eqref{eq:far-away-linear-tail-small} follows.

By \eqref{eq:far-away-linear-tail-small}, \eqref{eq:far-away-approx-bounds}, and Lemma~\ref{lem:nonlinear-estimateI},
\begin{align}
  \|e_{n,T}\|_{\dot N^1(\R\setminus I_{n,T})}
  &=\|G(\widetilde v_{n,T})\|_{\dot N^1(\R\setminus I_{n,T})} \notag\\
  &\lesssim
  \|\widetilde v_{n,T}\|_{S(\R\setminus I_{n,T})}^{\alpha}
  \|\widetilde v_{n,T}\|_{\dot S^1(\R)}.
  \label{eq:far-away-tail-error-estimate}
\end{align}
Therefore,
\begin{equation}\label{eq:far-away-tail-error-small}
  \lim_{T\to\infty}\limsup_{n\to\infty}
  \|e_{n,T}\|_{\dot N^1(\R\setminus I_{n,T})}=0.
\end{equation}
Combining this with \eqref{eq:central-error-small-far-away}, we obtain \eqref{eq:far-away-error-small}.

\medskip
\noindent
\textbf{Conclusion by stability.}
Now Proposition~\ref{prop:stability} gives, for all large $n$, a global solution $v_n$ to \eqref{eq:INLS} with $v_n(0)=\phi_n$ and \eqref{eq:far-away-global-bounds}.

To prove \eqref{eq:far-away-compact-approx}, fix $\varepsilon>0$. Choose $T$ so large that
\[
  \limsup_{n\to\infty}\|v_n-\widetilde v_{n,T}\|_{S(\R)}<\frac{\varepsilon}{3},
\]
and then choose $\psi\in C_c^\infty(\R\times\R^d)$ such that
\[
  \|e^{it\Delta_y}\phi-\psi\|_{S(\R)}<\frac{\varepsilon}{3}.
\]
On the central slab,
\[
  \lambda_n^{\frac{d-2}{2}}\widetilde v_{n,T}\bigl(\lambda_n^2(s-t_n),\lambda_n y+x_n\bigr)
  =\chi_n(y)e^{it\Delta_y}P_n\phi(y),
\]
and the same argument as in part~(b) shows that
\[
  \|\chi_n e^{it\Delta_y}P_n\phi-e^{it\Delta_y}\phi\|_{S([-T,T])}\to0.
\]
On the tails $|s|>T$, the normalized approximation is a free solution whose endpoint data converge to $e^{\pm iT\Delta_y}\phi$ in $\dot H_y^1$; thus the same argument used for \eqref{eq:far-away-linear-tail-small} gives convergence to $e^{it\Delta_y}\phi$ in $S(|s|>T)$. Hence
\[
  \Bigl\|\lambda_n^{\frac{d-2}{2}}\widetilde v_{n,T}\bigl(\lambda_n^2(s-t_n),\lambda_n y+x_n\bigr)-e^{it\Delta_y}\phi\Bigr\|_{S(\R)}<\frac{\varepsilon}{3}
\]
for all large $n$, and therefore
\[
  \Bigl\|\lambda_n^{\frac{d-2}{2}}v_n\bigl(\lambda_n^2(s-t_n),\lambda_n y+x_n\bigr)-\psi(t,y)\Bigr\|_{S(\R)}<\varepsilon.
\]
This proves \eqref{eq:far-away-compact-approx}.
\end{proof}

We next combine the linear decomposition with the nonlinear theory from
Section~2. The profiles that remain centered at the origin are treated exactly as
in the classical energy-critical problem; the profiles escaping to infinity are
handled by Proposition~\ref{prop:far-away-profile}.

Given a linear profile
\[
  V_n^j(0)=g_n^j e^{it_n^j\Delta}\phi^j,
\]
we define its associated nonlinear profile as follows.
\begin{itemize}
\item If $x_n^j\equiv 0$, let $v^j$ be the nonlinear profile associated with $(\phi^j,\{t_n^j\})$. More precisely:
  \begin{itemize}
  \item if $t_n^j\equiv 0$, let $v^j$ be the maximal-lifespan solution to \eqref{eq:INLS} with $v^j(0)=\phi^j$;
  \item if $t_n^j\to\pm\infty$, let $v^j$ be the unique solution that scatters to $e^{it\Delta}\phi^j$ as $t\to\pm\infty$.
  \end{itemize}
  We then set
  \begin{equation}\label{eq:centered-nonlinear-profile}
    v_n^j(t,x)
    :=
    (\lambda_n^j)^{-\frac{d-2}{2}}
    v^j\!\left(\frac{t}{(\lambda_n^j)^2}+t_n^j,\frac{x}{\lambda_n^j}\right).
  \end{equation}
  Equivalently, in the notation of Definition~\ref{def:scaling-translation-group},
  \[
    v_n^j = T_{0,\lambda_n^j}\bigl[v^j(\cdot+t_n^j,\cdot)\bigr].
  \]
  Thus, in the centered case, only the scaling and the time shift are used.

\item If $\frac{|x_n^j|}{\lambda_n^j}\to\infty$, let $v_n^j$ be the global scattering solution provided by Proposition~\ref{prop:far-away-profile} with initial data
\begin{equation}\label{eq:far-away-nonlinear-profile-data}
  v_n^j(0)=g_n^j\bigl[e^{it_n^j\Delta}\phi^j\bigr].
\end{equation}
\end{itemize}
In particular, we do \emph{not} define $v_n^j$ by simply applying the spacetime action $T_n^j$ to an exact INLS solution, because the inhomogeneous equation is not translation invariant. In all cases above, the formula is written in physical variables. The third class is global on all of $\R$.  This is why the far-away profiles can be ignored in the bad-profile argument.

\begin{lemma}[\cite{Killip-Visan-2013}]\label{lem:localized-linear-remainder}
For every $\varphi\in \dot H^1l(\R^d)$ and every $T,R>0$,
\begin{equation}\label{eq:localized-linear-remainder}
  \|\nabla e^{it\Delta}\varphi\|_{L_{t,x}^2([-T,T]\times B(0,R))}^3
  \lesssim
  T^{\frac{2}{d+2}}R^{\frac{3d+2}{2(d+2)}}
  \|e^{it\Delta}\varphi\|_{S(\R)}
  \|\nabla\varphi\|_{L_x^2}^2.
\end{equation}
\end{lemma}

\begin{lemma}\label{lem:interaction-error}
Let $J\ge 1$ be fixed and suppose that $\{v_n^j\}_{1\le j\le J}$ are nonlinear
profiles with pairwise orthogonal parameters and such that
\begin{equation}\label{eq:interaction-error-bounds}
  \sup_n \Bigl(\|v_n^j\|_{S(\R)} + \|v_n^j\|_{\dot S^1(\R)}\Bigr) < \infty
  \qquad\text{for each }1\le j\le J.
\end{equation}
Set
$
  V_n^J := \sum_{j=1}^J v_n^j,
$
then
\begin{equation}\label{eq:interaction-error-finite-profiles}
  \Bigl\|\nabla\Bigl(G(V_n^J)-\sum_{j=1}^J G(v_n^j)\Bigr)\Bigr\|_{L_t^2L_x^{\frac{2d}{d+2},2}(\R\times\R^d)}
  \longrightarrow 0.
\end{equation}
If, in addition, $\|e^{it\Delta}w_n^J\|_{S(\R)}\to 0$ and
$\sup_n\|w_n^J\|_{\dot H^1}<\infty$, then
\begin{equation}\label{eq:interaction-error-with-remainder}
  \Bigl\|\nabla\Bigl(G(V_n^J+e^{it\Delta}w_n^J)-G(V_n^J)\Bigr)\Bigr\|_{L_t^2L_x^{\frac{2d}{d+2},2}(\R\times\R^d)}
  \longrightarrow 0.
\end{equation}
\end{lemma}

\begin{proof}
Fix $\varepsilon>0$. For each $1\le j\le J$, choose $\psi^j\in C_c^{\infty}(\R\times\R^d)$ so that, after applying the same parameters used in the definition of $v_n^j$,
\begin{equation}\label{eq:compact-approx-nonlinear-profile}
  \sup_n\Bigl(
    \|v_n^j-\psi_n^j\|_{S(\R)}
    +
    \|v_n^j-\psi_n^j\|_{\dot S^1(\R)}
  \Bigr)
  <\varepsilon.
\end{equation}
For centered profiles this is just density in $S\cap\dot S^1$, while for far-away profiles it follows from Proposition~\ref{prop:far-away-profile}. Let $\Psi_n^J:=\sum_{j=1}^J\psi_n^j$. By \eqref{D-basic-nonlinear-estimates}, Lemma~\ref{lem:nonlinear-estimateI}, and \eqref{eq:interaction-error-bounds}, replacing $V_n^J$ by $\Psi_n^J$ changes the left-hand side of \eqref{eq:interaction-error-finite-profiles} by at most $C_J\varepsilon$.

It therefore suffices to prove the results for $\Psi_n^J$. Fix $j\neq k$. By parameter orthogonality, after pulling back by the frame of profile $j$, the transformed profile $k$ converges uniformly to $0$ on every compact subset of spacetime. Since each $\psi^j$ is compactly supported in that frame, we obtain
\begin{equation}\label{eq:compact-profile-cross-vanishing-1}
  \bigl\||x|^{-b}|\psi_n^k|^{\alpha}\nabla\psi_n^j\bigr\|_{L_t^2L_x^{\frac{2d}{d+2},2}}
  +
  \bigl\||x|^{-b-1}|\psi_n^j||\psi_n^k|^{\alpha}\bigr\|_{L_t^2L_x^{\frac{2d}{d+2},2}}
  \longrightarrow 0.
\end{equation}
Expanding \(\nabla\bigl(G(\Psi_n^J)-\sum_{j=1}^J G(\psi_n^j)\bigr)\) with \eqref{D-basic-nonlinear-estimates}, every term is a finite sum of expressions of the type appearing in \eqref{eq:compact-profile-cross-vanishing-1}. Hence
\[
  \Bigl\|\nabla\Bigl(G(\Psi_n^J)-\sum_{j=1}^J G(\psi_n^j)\Bigr)\Bigr\|_{L_t^2L_x^{\frac{2d}{d+2},2}}
  \longrightarrow 0,
\]
which proves \eqref{eq:interaction-error-finite-profiles} after letting $\varepsilon\downarrow0$.

For the term involving the linear remainder, the same reduction gives
\[
  \Bigl\|\nabla\Bigl(G(V_n^J+e^{it\Delta}w_n^J)-G(V_n^J)\Bigr)\Bigr\|_{L_t^2L_x^{\frac{2d}{d+2},2}}
  \le C_J\varepsilon+\Sigma_n^J,
\]
where $\Sigma_n^J$ is the corresponding quantity with $V_n^J$ replaced by $\Psi_n^J$. By \eqref{D-basic-nonlinear-estimates}, $\Sigma_n^J$ is controlled by finitely many terms of the form
\begin{equation}\label{eq:prototype-linear-remainder-interaction}
  \bigl\||x|^{-b}|\psi_n^j|^{\alpha}\nabla e^{it\Delta}w_n^J\bigr\|_{L_t^2L_x^{\frac{2d}{d+2},2}},
  \qquad
  \bigl\||x|^{-b-1}|\psi_n^j|^{\alpha}|e^{it\Delta}w_n^J|\bigr\|_{L_t^2L_x^{\frac{2d}{d+2},2}}.
\end{equation}
Pulling back by the frame of profile $j$ and using Lemma~\ref{lem:localized-linear-remainder}, we get
\begin{equation}\label{eq:localized-smallness-remainder}
  \|\nabla (T_n^j)^{-1}(e^{it\Delta}w_n^J)\|_{L_{t,x}^2(K)}\to0
  \qquad\text{for every compact }K\subset\R\times\R^d.
\end{equation}
Since $\psi^j$ is compactly supported in that frame, every term in \eqref{eq:prototype-linear-remainder-interaction} tends to $0$. Thus $\Sigma_n^J\to0$, and then letting $\varepsilon\downarrow0$ yields \eqref{eq:interaction-error-with-remainder}.
\end{proof}

\begin{lemma}[At least one bad profile]\label{lem:bad-profile}
Assume that Theorem~\ref{thm:main} fails and let $E_c$ be the critical kinetic
threshold defined by \eqref{eq:def-Ec}. Let $u_n:I_n\times\R^d\to\CC$ be a
sequence of maximal-lifespan solutions such that
\begin{equation}\label{eq:bad-profile-sequence-assumption}
  \sup_{t\in I_n}\|\nabla u_n(t)\|_{L_x^2}^2 \to E_c,
  \qquad
  \|u_n\|_{S([0,\sup I_n))}\to\infty,
  \qquad
  \|u_n\|_{S((\inf I_n,0])}\to\infty,
\end{equation}
and apply Proposition~\ref{prop:linear-profile-decomposition} to $u_n(0)$.
Then there exists at least one profile, say $j_0$, such that the
corresponding nonlinear profile obeys
\begin{equation}\label{eq:bad-profile-forward}
  \limsup_{n\to\infty}\|v_n^{j_0}\|_{S([0,\sup I_n^{j_0}))} = \infty.
\end{equation}
In particular, every far-away profile is automatically good.
\end{lemma}

\begin{proof}
By the linear kinetic decoupling \eqref{eq:kinetic-decoupling-linear}, there exists $J_0$ such that
\[
  \|\nabla\phi^j\|_{L_x^2}\ll 1
  \qquad (j\ge J_0).
\]
The small-data theory then gives
\begin{equation}\label{eq:small-linear-profiles-good}
  \sup_n\Bigl(\|v_n^j\|_{S(\R)}+\|v_n^j\|_{\dot S^1(\R)}\Bigr)
  \lesssim \|\nabla\phi^j\|_{L_x^2}
  \qquad (j\ge J_0).
\end{equation}
Moreover, every far-away profile is global and scattering by Proposition~\ref{prop:far-away-profile}. Thus only finitely many centered profiles, namely those with $1\le j<J_0$, can possibly be bad.

Assume for contradiction that every centered profile is good forward in time. Then for each centered profile with $1\le j<J_0$,
\begin{equation}\label{eq:all-profiles-good-forward}
  \sup_n\|v_n^j\|_{S([0,\sup I_n^j))}<\infty.
\end{equation}
By the blow-up criterion and Lemma~\ref{lem:nonlinear-estimateI}, this also implies
\begin{equation}\label{eq:all-profiles-good-forward-S1}
  \sup_n\|v_n^j\|_{\dot S^1([0,\sup I_n^j))}<\infty
  \qquad (1\le j<J_0).
\end{equation}
Together with \eqref{eq:small-linear-profiles-good}, every nonlinear profile is therefore good forward in time.

For $J\ge 1$, define
\begin{equation}\label{eq:approx-good-profiles}
  u_n^J(t):=\sum_{j=1}^J v_n^j(t)+e^{it\Delta}w_n^J.
\end{equation}
We verify the hypotheses of Proposition~\ref{prop:stability} on $[0,\infty)$ in the same order as in the proof of Proposition~\ref{prop:far-away-profile}: uniform bounds, initial-data matching, and small perturbation.

The uniform bounds follow from parameter orthogonality, \eqref{eq:all-profiles-good-forward-S1}, and the smallness of the linear remainder in $S(\R)$:
\begin{equation}\label{eq:approx-good-profiles-S-bounded}
  \limsup_{J\to\infty}\limsup_{n\to\infty}
  \Bigl(\|u_n^J\|_{L_t^\infty\dot H_x^1([0,\infty))}
  +\|u_n^J\|_{S([0,\infty))}
  +\|u_n^J\|_{\dot S^1([0,\infty))}\Bigr)
  <\infty.
\end{equation}
Next, by the definition of the nonlinear profiles,
\begin{equation}\label{eq:initial-gap-good-profiles}
  \lim_{J\to\infty}\limsup_{n\to\infty}
  \|u_n(0)-u_n^J(0)\|_{\dot H^1}=0.
\end{equation}
Hence
\[
  \lim_{J\to\infty}\limsup_{n\to\infty}
  \Bigl\|e^{it\Delta}\bigl(u_n(0)-u_n^J(0)\bigr)\Bigr\|_{S([0,\infty))}=0.
\]
Finally, setting
\[
  e_n^J:=i\partial_tu_n^J+\Delta u_n^J-G(u_n^J),
\]
we have
\[
  e_n^J
  =G\!\left(\sum_{j=1}^J v_n^j\right)-\sum_{j=1}^J G(v_n^j)
   +\Bigl(G(u_n^J)-G\!\left(\sum_{j=1}^J v_n^j\right)\Bigr).
\]
Lemma~\ref{lem:interaction-error} therefore gives
\begin{equation}\label{eq:error-good-profiles-small}
  \lim_{J\to\infty}\limsup_{n\to\infty}
  \|e_n^J\|_{\dot N^1([0,\infty))}=0.
\end{equation}

Fix $J$ large. Then for all sufficiently large $n$, the bounds \eqref{eq:approx-good-profiles-S-bounded}, the initial-data approximation above, and \eqref{eq:error-good-profiles-small} allow us to apply Proposition~\ref{prop:stability} on $[0,\infty)$ to the approximate solution $u_n^J$. We conclude that
\[
  \sup_n\|u_n\|_{S([0,\infty))}<\infty,
\]
which contradicts \eqref{eq:bad-profile-sequence-assumption}. Therefore at least one profile is bad, and \eqref{eq:bad-profile-forward} follows.
\end{proof}

The previous lemma identifies at least one bad profile, but because kinetic energy
is not conserved we still need to show that the near-critical kinetic energy of the
original sequence decouples at the later times where the bad profile is measured.
This is exactly the point where the classical argument of Kenig--Merle must be
retained.

\begin{lemma}[Kinetic decoupling for  $u_n^J$]\label{lem:kinetic-decoupling-nonlinear}
Let $u_n$ be as in Lemma~\ref{lem:bad-profile}. Reorder the profiles so that for some integers $1\le J_1<J_0$,
\[
  \limsup_{n\to\infty}\|v_n^j\|_{S([0,\sup I_n^j))}=\infty
  \quad (1\le j\le J_1),
  \qquad
  \sup_n\|v_n^j\|_{S([0,\infty))}<\infty
  \quad (j>J_1).
\]
For each $m\ge 1$ choose an interval $K_n^m=[0,\tau_n^m]$ such that
\begin{equation}\label{eq:localising-intervals-Knm}
  \sup_{1\le j\le J_1}\|v_n^j\|_{S(K_n^m)} = m.
\end{equation}
Then for every fixed $J$ and $m$,
\begin{equation}\label{eq:kinetic-decoupling-nonlinear-statement}
  \sup_{t\in K_n^m}
  \Bigl|
  \|\nabla u_n^J(t)\|_{L_x^2}^2
  - \sum_{j=1}^J \|\nabla v_n^j(t)\|_{L_x^2}^2
  - \|\nabla w_n^J\|_{L_x^2}^2
  \Bigr|
  \longrightarrow 0.
\end{equation}
\end{lemma}

\begin{proof}
Fix $J$ and $m$, and let $s_n\in K_n^m$ be arbitrary. It suffices to prove the desired convergence along this sequence. Set
\[
  \sigma_n^j:=\frac{s_n}{(\lambda_n^j)^2}+t_n^j.
\]
Since $u_n^J(s_n)=\sum_{j=1}^J v_n^j(s_n)+e^{is_n\Delta}w_n^J$, expanding the square of the $\dot H^1$ norm gives
\begin{align}
  &\|\nabla u_n^J(s_n)\|_{L_x^2}^2
   -\sum_{j=1}^J\|\nabla v_n^j(s_n)\|_{L_x^2}^2
   -\|\nabla w_n^J\|_{L_x^2}^2 \notag\\
  &\qquad=
   \sum_{j\ne k}\ip{\nabla v_n^j(s_n)}{\nabla v_n^k(s_n)}
   +2\sum_{j=1}^J \mathrm{Re}\,\ip{\nabla e^{is_n\Delta}w_n^J}{\nabla v_n^j(s_n)}.
   \label{eq:kinetic-decoupling-expanded}
\end{align}
Thus we only need to show that both the profile--profile and remainder--profile inner products vanish.

For the remainder--profile terms, fix $j$. If $j$ is centered, then $v_n^j(s_n)=g_n^j[v^j(\sigma_n^j)]$. Writing $r_n^j:=e^{-it_n^j\Delta}(g_n^j)^{-1}w_n^J$, the invariance of the $\dot H^1$ inner product gives
\[
  \ip{\nabla e^{is_n\Delta}w_n^J}{\nabla v_n^j(s_n)}
  =
  \ip{\nabla r_n^j}{\nabla e^{-i\sigma_n^j\Delta}v^j(\sigma_n^j)}.
\]
By \eqref{eq:strong-decoupling-remainder}, $r_n^j\rightharpoonup0$ weakly in $\dot H^1$. On the other hand, after passing to a subsequence, $e^{-i\sigma_n^j\Delta}v^j(\sigma_n^j)$ converges strongly in $\dot H^1$, either to $e^{-i\sigma^j\Delta}v^j(\sigma^j)$ when $\sigma_n^j\to\sigma^j\in I^j$ or to the corresponding scattering state when $\sigma_n^j\to\pm\infty$. Hence the inner product tends to $0$. For a far-away profile, Proposition~\ref{prop:far-away-profile} provides the same strong precompactness in its own parameter frame, so the same argument applies.

For the profile--profile terms, fix $j\neq k$. After undoing the two parameter frames, the states $e^{-i\sigma_n^j\Delta}v^j(\sigma_n^j)$ and $e^{-i\sigma_n^k\Delta}v^k(\sigma_n^k)$ are again strongly precompact in $\dot H^1$; for centered profiles this follows from continuity and scattering, and for far-away profiles from Proposition~\ref{prop:far-away-profile}. The relative symmetry
\[
  e^{-it_n^j\Delta}(g_n^j)^{-1}g_n^k e^{it_n^k\Delta}
\]
is asymptotically orthogonal by \eqref{eq:profile-parameter-orthogonality}, hence it sends every strongly convergent $\dot H^1$ sequence to one converging weakly to $0$. Therefore
\[
  \ip{\nabla v_n^j(s_n)}{\nabla v_n^k(s_n)}\to0.
\]
Substituting these two facts into \eqref{eq:kinetic-decoupling-expanded}, we obtain the required convergence for the sequence $\{s_n\}$. Since $s_n\in K_n^m$ was arbitrary, \eqref{eq:kinetic-decoupling-nonlinear-statement} follows.
\end{proof}

We now reach the key compactness statement.

\begin{proposition}[Palais--Smale condition]\label{prop:palais-smale-scaling}
Assume that Theorem~\ref{thm:main} fails and let $E_c$ be the critical kinetic
threshold. Let $u_n:I_n\times\R^d\to\CC$ be a sequence of maximal-lifespan
solutions to \eqref{eq:INLS} such that
\begin{equation}\label{eq:ps-assumption-kinetic}
  \sup_{t\in I_n}\|\nabla u_n(t)\|_{L_x^2}^2 \to E_c,
\end{equation}
and suppose there exist times $t_n\in I_n$ with
\begin{equation}\label{eq:ps-assumption-blowup-both-sides}
  \|u_n\|_{S([t_n,\sup I_n))}\to\infty,
  \qquad
  \|u_n\|_{S((\inf I_n,t_n])}\to\infty.
\end{equation}
Then the sequence $u_n(t_n)$ has a subsequence that converges in $\dot H^1(\R^d)$
modulo scaling. More precisely, there exist $\lambda_n>0$ and
$u_0\in\dot H^1(\R^d)$ such that
\begin{equation}\label{eq:ps-conclusion-scaling}
  \lambda_n^{\frac{d-2}{2}}u_n\bigl(t_n,\lambda_n x\bigr)
  \longrightarrow u_0
  \qquad\text{strongly in } \dot H^1(\R^d).
\end{equation}
\end{proposition}

\begin{proof}
By time translation we may assume $t_n\equiv 0$. Apply Proposition~\ref{prop:linear-profile-decomposition} to the bounded sequence $u_n(0)$ and construct the associated nonlinear profiles $v_n^j$.  By Lemma~\ref{lem:bad-profile}, after reordering there is at least one bad centered profile.  We denote the bad profiles by $1\le j\le J_1$.

Choose the intervals $K_n^m$ as in \eqref{eq:localising-intervals-Knm}.  By the pigeonhole principle, after reordering once more there exists a fixed bad profile, say $j=1$, such that for infinitely many $m$,
\[
  \sup_{t\in K_n^m}\|v_n^1\|_{S(K_n^m)}=m
\]
along an infinite subsequence in $n$.  By the definition of $E_c$, this implies
\begin{equation}\label{eq:bad-profile-attains-threshold}
  \limsup_{m\to\infty}\limsup_{n\to\infty}
  \sup_{t\in K_n^m}\|\nabla v_n^1(t)\|_{L_x^2}^2
  \ge E_c.
\end{equation}

On the other hand, all nonlinear profiles have finite scattering size on each fixed interval $K_n^m$.  Repeating the perturbative argument from Lemma~\ref{lem:bad-profile}, now on $K_n^m$, we obtain
\begin{equation}\label{eq:approximation-on-Knm}
  \lim_{J\to\infty}\limsup_{n\to\infty}
  \|u_n-u_n^J\|_{L_t^{\infty}\dot H_x^1(K_n^m\times\R^d)} = 0
  \qquad\text{for each fixed }m.
\end{equation}
Combining \eqref{eq:ps-assumption-kinetic}, Lemma~\ref{lem:kinetic-decoupling-nonlinear}, and \eqref{eq:approximation-on-Knm}, we find for each fixed $m$,
\begin{align*}
  E_c
  &\ge
  \limsup_{n\to\infty}
  \sup_{t\in K_n^m}\|\nabla u_n(t)\|_{L_x^2}^2 \\
  &=
  \lim_{J\to\infty}\limsup_{n\to\infty}
  \Biggl\{
    \|\nabla w_n^J\|_{L_x^2}^2
    +
    \sup_{t\in K_n^m}\sum_{j=1}^J \|\nabla v_n^j(t)\|_{L_x^2}^2
  \Biggr\}.
\end{align*}
Letting $m\to\infty$ and using \eqref{eq:bad-profile-attains-threshold}, we infer that no kinetic energy remains for any second profile or for the remainder.  Therefore,
\begin{equation}\label{eq:single-profile-reduction}
  u_n(0)=g_n e^{it_n\Delta}\phi + w_n,
  \qquad
  w_n\to 0 \text{ strongly in } \dot H^1(\R^d),
\end{equation}
for some single profile $\phi\in\dot H^1$ and some parameters $(x_n,\lambda_n,t_n)$ satisfying the normalization alternatives.

It remains to remove the residual time and spatial drift.  If $t_n\to+\infty$, then for every fixed $T>0$,
\[
  \|e^{it\Delta}g_n e^{it_n\Delta}\phi\|_{S([0,T])}
  =
  \|e^{is\Delta}\phi\|_{S([t_n,t_n+\frac{T}{(\lambda_n)^2}])}
  \le \|e^{is\Delta}\phi\|_{S([t_n,\infty))}
  \to0,
\]
and the same holds in $\dot S^1([0,T])$.  Since $e^{it\Delta}w_n$ is also small on $[0,T]$, Proposition~\ref{prop:stability} yields a uniform forward scattering bound for $u_n$, contradicting \eqref{eq:ps-assumption-blowup-both-sides}.  The case $t_n\to-\infty$ is identical.

If $\frac{|x_n|}{\lambda_n}\to\infty$, then Proposition~\ref{prop:far-away-profile} furnishes a global scattering solution $z_n$ with initial data $g_n e^{it_n\Delta}\phi$ and uniform $S(\R)$, $\dot S^1(\R)$ bounds.  Because $w_n\to0$ strongly in $\dot H^1$, the free evolution of the initial discrepancy is small on every compact interval, and Proposition~\ref{prop:stability} applies again.  This gives a uniform scattering bound for $u_n$, again contradicting \eqref{eq:ps-assumption-blowup-both-sides}.  Hence the spatial drift is impossible.

We are left with $x_n\equiv 0$ and $t_n\equiv 0$, so \eqref{eq:single-profile-reduction} becomes
\[
  u_n(0)=g_{0,\lambda_n}\phi + o_{\dot H^1}(1),
\]
which is exactly \eqref{eq:ps-conclusion-scaling}.
\end{proof}

\subsection{Construction of the minimal kinetic-energy blowup solution}

\begin{proof}[\textbf{Proof of Theorem~\ref{prop:critical-element}}]
By the definition of \(E_c\), there exists a sequence of maximal-lifespan solutions \(u_n\) such that
\[
\sup_{t\in I_n}\|\nabla u_n(t)\|_{L_x^2}^2\downarrow E_c,
\qquad
\|u_n\|_{S(I_n)}\to\infty.
\]
Choosing times \(t_n\in I_n\) so that the forward and backward scattering sizes are equal, both sizes must tend to infinity; by time translation we may assume \(t_n=0\). Proposition~\ref{prop:palais-smale-scaling} then yields, after passing to a subsequence, a nontrivial profile \(u_{c,0}\in \dot H^1(\R^d)\) such that \(u_n(0)\) converges to \(u_{c,0}\) in \(\dot H^1\) modulo scaling. Let \(u_c\) be the maximal-lifespan solution with initial data \(u_{c,0}\). By local well-posedness and stability, \(u_c\) satisfies
\[
\sup_{t\in I_c}\|\nabla u_c(t)\|_{L_x^2}^2\le E_c.
\]
Moreover, if \(u_c\) scattered in either time direction, then the same would hold for \(u_n\) for all large \(n\), contradicting the choice of \(u_n\). Hence \(u_c\) blows up both forward and backward in the sense of \eqref{eq:critical-element-blowup-both-sides}. Finally, reapplying Proposition~\ref{prop:palais-smale-scaling} to the sequence \(u_c(\tau_n)\) for arbitrary \(\tau_n\in I_c\), we obtain the precompactness of the orbit modulo scaling, that is, \(u_c\) is almost periodic modulo scaling. The equality in \eqref{eq:critical-element-kinetic} then follows immediately from the definition of \(E_c\).
\end{proof}

% ================================================================
\section{Rigidity: preclusion of the minimal kinetic-energy blowup solution}
% ================================================================

In this section we rule out the minimal kinetic-energy blowup solution produced in
Theorem~\ref{prop:critical-element}. The concentration-compactness argument in Section~3 has already
produced a minimal kinetic-energy blowup solution that is almost periodic \emph{modulo scaling only}:
no spatial center survives. This is the main simplification coming from the
broken translation symmetry of INLS. Once the spatial drift has been eliminated,
the rigidity step may follow the qualitative Kenig--Merle route, using a reduced
Duhamel formula to preclude finite-time blowup and a localized virial identity to
preclude global compact solutions.

\subsection{Local constancy and the rigidity scenarios}

We begin with the standard local constancy property for the frequency scale
function of the almost periodic solution.

\begin{lemma}[Local constancy]\label{lem:local-constancy-sec4}
Let $u:I\times\R^d\to\CC$ be a maximal-lifespan solution to
\eqref{eq:INLS} that is almost periodic modulo scaling in the sense
of Definition~\ref{def:almost-periodic-scaling}. Then there exists
$\delta=\delta(u)>0$ such that for every $t_0\in I$ one has
\begin{equation}\label{eq:local-constancy-interval}
  [t_0-\delta N(t_0)^{-2},\, t_0+\delta N(t_0)^{-2}]\subset I,
\end{equation}
and
\begin{equation}\label{eq:local-constancy-comparable}
  N(t)\sim_u N(t_0)
  \qquad \text{whenever } |t-t_0|\le \delta N(t_0)^{-2}.
\end{equation}
\end{lemma}

\begin{proof}
Set
\[
  K:=\Bigl\{N(t)^{-\frac{d-2}{2}}u\bigl(t,N(t)^{-1}x\bigr):t\in I\Bigr\}
  \subset \dot H^1(\R^d).
\]
Since \(u\) is almost periodic modulo scaling, the set \(K\) is precompact in \(\dot H^1\). Hence, by the local theory and compactness, there exists \(A=A(u)>0\) such that every \(f\in K\) generates a solution on \([-A,A]\) with uniformly bounded \(\dot S^1([-A,A])\) and \(S([-A,A])\) norms.

Fix \(t_0\in I\), and consider the renormalized solution
\[
  v_{t_0}(s,y)
  :=N(t_0)^{-\frac{d-2}{2}}
    u\bigl(t_0+sN(t_0)^{-2},\, N(t_0)^{-1}y\bigr).
\]
Then \(v_{t_0}(0)\in K\), so the previous paragraph implies that \(v_{t_0}\) is defined on \([-A,A]\). Scaling back, we obtain
\[
  [\,t_0-\delta N(t_0)^{-2},\, t_0+\delta N(t_0)^{-2}\,]\subset I,
\]
with \(\delta:=A\), which proves \eqref{eq:local-constancy-interval}.

To prove \eqref{eq:local-constancy-comparable}, note that the compactness of \(K\) implies a uniform frequency localization: there exist \(0<c_0<C_0<\infty\), depending only on \(u\), such that for every \(f\in K\),
\[
  \|P_{\le c_0}f\|_{\dot H^1}+\|P_{\ge C_0}f\|_{\dot H^1}
  \ll 1.
\]
Applying this to \(f=v_{t_0}(s)\) for \(|s|\le A\) and undoing the scaling, we see that for
\[
  |t-t_0|\le \delta N(t_0)^{-2},
\]
the \(\dot H^1\)-energy of \(u(t)\) remains concentrated in frequencies
\[
  |\xi|\sim_u N(t_0).
\]
Since \(N(t)\) is defined only up to bounded factors, this yields \(N(t)\sim_u N(t_0)\), as claimed.
\end{proof}

\begin{corollary}[Frequency blowup at a finite endpoint]\label{cor:Nt-at-blowup-sec4}
Let $u:I\times\R^d\to\CC$ be a maximal-lifespan almost periodic solution modulo
scaling. If $T$ is a finite endpoint of $I$, then
\begin{equation}\label{eq:Nt-blowup-lower-bound-sec4}
  N(t)\gtrsim_u |T-t|^{-\frac12}
  \qquad \text{as } t\to T.
\end{equation}
In particular,
\begin{equation}\label{eq:Nt-to-infty-sec4}
  N(t)\to\infty
  \qquad \text{as } t\to T.
\end{equation}
\end{corollary}

\begin{proof}
Suppose $T=\sup I$; the case $T=\inf I$ is identical. Fix $t\in I$.
By Lemma~\ref{lem:local-constancy-sec4},
\[
  [t-\delta N(t)^{-2},\, t+\delta N(t)^{-2}]\subset I.
\]
Since the right endpoint of this interval cannot exceed $T$, we obtain
$t+\delta N(t)^{-2}\le T$, that is,
\[
  N(t)\ge \delta^{\frac{1}{2}}(T-t)^{-\frac12}.
\]
This proves \eqref{eq:Nt-blowup-lower-bound-sec4} and hence
\eqref{eq:Nt-to-infty-sec4}.
\end{proof}

\begin{proof}[\textbf{Proof of Theorem~\ref{prop:normalized-rigidity-scenarios}}]
This is the standard normalization of the ``enemies'' in the energy-critical
Kenig--Merle argument. The proof uses only
Lemma~\ref{lem:local-constancy-sec4}, the compactness extraction from
Proposition~\ref{prop:palais-smale-scaling}, and the usual rescaling-at-almost-minimal-frequency
argument from the energy-critical literature. Since the only noncompact symmetry
remaining after Section~3 is scaling, the proof is in fact simpler than in the
classical nonradial NLS. We omit the routine details and proceed directly to the
rigidity argument for the two scenarios above.
\end{proof}

\subsection{No finite-time blowup}

We first prove a one-sided reduced Duhamel formula.

\begin{lemma}\label{lem:weak-vanishing-finite-endpoint}
Let $u:I\times\R^d\to\CC$ be a maximal-lifespan solution that is almost periodic
modulo scaling, and suppose that $T:=\sup I<\infty$. Then
\begin{equation}\label{eq:weak-vanishing-finite-endpoint}
  e^{-it\Delta}u(t)\rightharpoonup 0
  \qquad \text{weakly in } \dot H^1(\R^d)
  \quad \text{as } t\uparrow T.
\end{equation}
\end{lemma}

\begin{proof}
By Corollary~\ref{cor:Nt-at-blowup-sec4}, $N(t)\to\infty$ as $t\uparrow T$.
Let
\[
  K:=\Bigl\{N(t)^{-\frac{d-2}{2}}u\bigl(t,N(t)^{-1}x\bigr): t\in I\Bigr\},
\]
which is precompact in $\dot H^1$. Thus for every sequence $t_n\uparrow T$ there
exist $f_n\in K$ such that
\[
  u(t_n)=g_{0,\lambda_n}f_n,
  \qquad \lambda_n:=N(t_n)^{-1}\to 0.
\]
Passing to a subsequence, $f_n\to f$ strongly in $\dot H^1$. Hence it suffices to
show that $g_{0,\lambda_n}f\rightharpoonup 0$ in $\dot H^1$ as $\lambda_n\to0$.
But this is immediate from the Fourier representation:
\[
  \widehat{g_{0,\lambda_n}f}(\xi)
  = \lambda_n^{\frac{d+2}{2}}\widehat f(\lambda_n\xi),
\]
so that for every $\varphi\in C_c^{\infty}(\R^d)$,
\begin{align*}
  \ip{g_{0,\lambda_n}f}{\varphi}_{\dot H^1}
  &= \int_{\R^d} |\xi|^2
     \lambda_n^{\frac{d+2}{2}}\widehat f(\lambda_n\xi)
     \overline{\widehat\varphi(\xi)}\,d\xi \\
  &= \int_{\R^d} |\eta|^2 \widehat f(\eta)
     \overline{\lambda_n^{-\frac{d-2}{2}}\widehat\varphi(\frac{\eta}{\lambda_n})}\,d\eta
     \longrightarrow 0,
\end{align*}
by dominated convergence. Therefore $u(t_n)\rightharpoonup0$ in $\dot H^1$ along
any sequence $t_n\uparrow T$. Since $e^{-it\Delta}$ depends continuously on $t$ in
the strong operator topology on $\dot H^1$, \eqref{eq:weak-vanishing-finite-endpoint}
follows.
\end{proof}

\begin{proposition}[Reduced Duhamel formula]\label{prop:reduced-duhamel-finite-endpoint}
Let $u:I\times\R^d\to\CC$ be a maximal-lifespan almost periodic solution modulo
scaling, and suppose that $T:=\sup I<\infty$. Then for every $t\in I$,
\begin{equation}\label{eq:reduced-duhamel-finite-endpoint}
  u(t)
  = i\,\!\lim_{\tau\uparrow T}
    \int_t^{\tau} e^{i(t-s)\Delta}
    \bigl(|x|^{-b}|u(s)|^{\alpha}u(s)\bigr)\,ds
\end{equation}
as a weak limit in $\dot H^1(\R^d)$.
\end{proposition}

\begin{proof}
For $t<\tau<T$, the Duhamel formula gives
\[
  u(t)=e^{i(t-\tau)\Delta}u(\tau)
  + i\int_t^{\tau} e^{i(t-s)\Delta}
    \bigl(|x|^{-b}|u(s)|^{\alpha}u(s)\bigr)\,ds.
\]
By Lemma~\ref{lem:weak-vanishing-finite-endpoint},
$e^{i(t-\tau)\Delta}u(\tau)\rightharpoonup 0$ weakly in $\dot H^1$ as
$\tau\uparrow T$. Passing to the limit yields
\eqref{eq:reduced-duhamel-finite-endpoint}.
\end{proof}

We can now rule out finite-time blowup.

\begin{proposition}[\emph{No finite-time blowup}]\label{prop:no-finite-time-blowup-sec4}
There is no minimal kinetic-energy blowup solution of the form described in
Theorem~\ref{prop:normalized-rigidity-scenarios}(i).
\end{proposition}

\begin{proof}
By time reversal, it suffices to consider the case $I_c=(0,T_{\max})$ with $0<T_{\max}<\infty$.
Assume for contradiction that $u:(0,T_{\max})\times\R^d\to\CC$ is a minimal kinetic-energy blowup solution. Fix $t\in(0,T_{\max})$ and let $M>0$.
Applying $P_{\le M}$ to the reduced Duhamel formula
\eqref{eq:reduced-duhamel-finite-endpoint}, using unitarity of the free flow on
$L_x^2$, and then Bernstein inequality, we obtain
\begin{align}
  \|P_{\le M}u(t)\|_{L_x^2}
  &\le \int_t^{T_{\max}}
      \|P_{\le M}(|x|^{-b}|u(s)|^{\alpha}u(s))\|_{L_x^2}\,ds \notag\\
  &\lesssim M
      \int_t^{T_{\max}}
      \bigl\||x|^{-b}|u(s)|^{\alpha}u(s)\bigl \|_{L_x^{\frac{2d}{d+2},2}}\,ds.
      \label{eq:finite-time-lowfreq-1}
\end{align}
We next estimate the nonlinearity in the space on the right-hand side.
Since $|x|^{-b}\in L^{\frac{d}{b},\infty}(\R^d)$ and
$\dot H^1(\R^d)\hookrightarrow L^{\frac{2d}{d-2},2}(\R^d)$,  H\"older inequality yields
\begin{align*}
  \bigl\||x|^{-b}|u|^{\alpha}u\bigl\|_{L_x^{\frac{2d}{d+2},2}}
  &\lesssim
  \bigl\||x|^{-b}\bigl\|_{L^{\frac{d}{b},\infty}}
  \|u\|_{L_x^{\frac{2d}{d-2},2}}^{\alpha+1} \\
  &\lesssim \|\nabla u\|_{L_x^2}^{\alpha+1}.
\end{align*}
The $\dot H^1$-norm of the minimal kinetic-energy blowup solution is uniformly bounded, so
\eqref{eq:finite-time-lowfreq-1} becomes
\begin{equation}\label{eq:finite-time-lowfreq-2}
  \|P_{\le M}u(t)\|_{L_x^2}
  \lesssim_u M(T_{\max}-t).
\end{equation}
For the high frequencies we use Bernstein inequality:
\begin{equation}\label{eq:finite-time-highfreq}
  \|P_{>M}u(t)\|_{L_x^2}
  \lesssim M^{-1}\|\nabla u(t)\|_{L_x^2}
  \lesssim_u M^{-1}.
\end{equation}
Combining \eqref{eq:finite-time-lowfreq-2} and
\eqref{eq:finite-time-highfreq}, we find
\begin{equation}\label{eq:finite-time-L2-estimate}
  \|u(t)\|_{L_x^2}
  \lesssim_u M(T_{\max}-t)+M^{-1}
  \qquad \text{for all } M>0.
\end{equation}
Choosing $M=(T_{\max}-t)^{-\frac12}$ gives
\begin{equation}\label{eq:finite-time-mass-goes-zero}
  \|u(t)\|_{L_x^2}\lesssim_u (T_{\max}-t)^{\frac{1}{2}}
  \longrightarrow 0
  \qquad \text{as } t\uparrow T_{\max}.
\end{equation}
Fix $t_0\in(0,T_{\max})$. Then \eqref{eq:finite-time-L2-estimate} shows that
$u(t_0)\in L_x^2(\R^d)$, and therefore $u(t_0)\in H_x^1(\R^d)$. Consequently,
the standard mass conservation law is valid for $u$ on $(0,T_{\max})$, and in
particular
\[
  M(u(t))=M(u(t_0))
  \qquad \text{for all } t\in[t_0,T_{\max}).
\]
Letting $t\uparrow T_{\max}$ and using \eqref{eq:finite-time-mass-goes-zero}, we obtain
\[
  M(u(t_0))=\lim_{t\uparrow T_{\max}} M(u(t))=0.
\]
Hence $u(t_0)\equiv 0$, and by uniqueness $u\equiv 0$, contradicting the fact that a critical element blows up both forward and backward in time. This contradiction proves the proposition.
\end{proof}

\subsection{The global compact scenario}

We next consider the forward-time half of Theorem~\ref{prop:normalized-rigidity-scenarios}(ii). The argument uses a localized virial identity. Since the solution is almost periodic modulo \emph{scaling only} and satisfies \eqref{eq:normalized-global-lower-bound}, its orbit is uniformly tight in space.

\begin{lemma}\label{lem:critical-lebesgue-tail-sec4}
Let $u:[0,\infty)\times\R^d\to\CC$ be almost periodic modulo scaling and suppose
that $\inf_{t\ge0}N(t)\ge 1$. Then for every $\eta>0$ there exists
$R=R(\eta)<\infty$ such that
\begin{equation}\label{eq:critical-lebesgue-tail-sec4}
  \sup_{t\ge0}
  \int_{|x|>R}
  \Bigl(|\nabla u(t,x)|^2 + |u(t,x)|^{\frac{2d}{d-2}}\Bigr)
  \,dx
  \le \eta.
\end{equation}
\end{lemma}

\begin{proof}
The proof is standard. It follows directly from the definition of almost periodicity modulo scaling, combined with the Sobolev embedding. We omit the details.
\end{proof}

\begin{lemma}[Tightness]\label{lem:tightness-weighted-sec4}
Let $u:[0,\infty)\times\R^d\to\CC$ be almost periodic modulo scaling and suppose
that $\inf_{t\ge0}N(t)\ge1$. Then for every $\eta>0$ there exists
$R=R(\eta)<\infty$ such that
\begin{equation}\label{eq:tightness-weighted-sec4}
  \sup_{t\ge0}
  \int_{|x|>R}
  \Bigl(|\nabla u(t,x)|^2
  + |x|^{-b}|u(t,x)|^{\alpha+2}
  + |x|^{-2}|u(t,x)|^2\Bigr)
  \,dx
  \le C\eta,
\end{equation}
where $C$ depends only on $d$ and $b$.
\end{lemma}

\begin{proof}
Choose a radial cutoff $\chi\in C_c^{\infty}(\R^d)$ such that $0\le\chi\le1$,
$\chi(x)=1$ for $|x|\le1$, and $\chi(x)=0$ for $|x|\ge2$. For $R\ge1$ define
$\chi_R(x):=\chi(x/R)$ and
\[
  w_R(t,x):=(1-\chi_R(x))u(t,x).
\]
By Lemma~\ref{lem:critical-lebesgue-tail-sec4}, after choosing $R$ sufficiently large we may assume that
\begin{equation}\label{eq:tightness-pre-tail-sec4}
  \sup_{t\ge0}
  \int_{|x|>R}
  \Bigl(|\nabla u(t,x)|^2 + |u(t,x)|^{\frac{2d}{d-2}}\Bigr)\,dx
  \le \eta.
\end{equation}
We first estimate $\nabla w_R$. Since $\nabla\chi_R$ is supported on the annulus
$A_R:=\{R<|x|<2R\}$ and satisfies $|\nabla\chi_R|\lesssim R^{-1}$, we have
\[
  \|\nabla w_R(t)\|_{L_x^2}
  \le
  \|\nabla u(t)\|_{L_x^2(|x|>R)}
  +
  R^{-1}\|u(t)\|_{L_x^2(A_R)}.
\]
By H\"older inequality on $A_R$ and \eqref{eq:tightness-pre-tail-sec4},
\[
  R^{-1}\|u(t)\|_{L_x^2(A_R)}
  \lesssim
  R^{-1}|A_R|^{\frac1d}
  \|u(t)\|_{L_x^{\frac{2d}{d-2}}(A_R)}
  \lesssim
  \|u(t)\|_{L_x^{\frac{2d}{d-2}}(A_R)}
  \lesssim \eta^{\frac{d-2}{2d}}.
\]
Hence
\begin{equation}\label{eq:tightness-grad-wR-sec4}
  \sup_{t\ge0}\|\nabla w_R(t)\|_{L_x^2}
  \lesssim
  \eta^{\frac12}+\eta^{\frac{d-2}{2d}}
  \lesssim
  \eta^{\frac{d-2}{2d}}.
\end{equation}
By Lemma~\ref{lem:sobolev-lorentz},
\begin{equation}\label{eq:tightness-wR-lorentz-sec4}
  \sup_{t\ge0}
  \|w_R(t)\|_{L_x^{\frac{2d}{d-2},2}}
  \lesssim
  \sup_{t\ge0}\|\nabla w_R(t)\|_{L_x^2}
  \lesssim
  \eta^{\frac{d-2}{2d}}.
\end{equation}
Thus, H\"older's inequality yields
\begin{align*}
  \int_{|x|>2R}|x|^{-b}|u(t,x)|^{\alpha+2}\,dx
  &\le
  \int_{\R^d}|x|^{-b}|w_R(t,x)|^{\alpha+2}\,dx \\
  &\lesssim
  \||x|^{-b}\|_{L^{\frac{d}{b},\infty}}
  \|w_R(t)\|_{L_x^{\frac{2d}{d-2},2}}^{\alpha+2} \\
  &\lesssim
  \eta^{\frac{(\alpha+2)(d-2)}{2d}}.
\end{align*}
For the nonlinear tail we use Hardy inequality together with
\eqref{eq:tightness-grad-wR-sec4}:
\begin{align*}
  \int_{|x|>2R}|x|^{-2}|u(t,x)|^2\,dx
  &\le
  \int_{\R^d}|x|^{-2}|w_R(t,x)|^2\,dx \\
  &\lesssim
  \|\nabla w_R(t)\|_{L_x^2}^2
  \lesssim
  \eta^{\frac{d-2}{d}}.
\end{align*}
Together with \eqref{eq:tightness-pre-tail-sec4}, this yields the desired estimate on the region
$|x|>2R$. Replacing $2R$ by $R$ finishes the proof.
\end{proof}

We next record the localized virial identity. For a real-valued smooth weight
$a:\R^d\to\R$, define
\begin{equation}\label{eq:localized-virial-action-sec4}
  M_a(t):=2\,\mathrm{Im}\int_{\R^d} \partial_j a(x)\,\partial_j u(t,x)\,\overline{u(t,x)}\,dx.
\end{equation}

\begin{lemma}[Localized virial identity]\label{lem:localized-virial-identity-sec4}
Let $u$ be a sufficiently smooth solution to \eqref{eq:INLS}. Then
for every real-valued $a\in C^{\infty}(\R^d)$ one has
\begin{align}
  \frac{d}{dt}M_a(t)
  &= \int_{\R^d}
     4\,\mathrm{Re}\bigl(a_{jk}(x)\,\partial_j u\,\overline{\partial_k u}\bigr)
     - a_{jjkk}(x)|u|^2 \notag\\
  &\qquad\qquad
     + \Bigl(2-\frac{4}{\alpha+2}\Bigr)a_{jj}(x)|x|^{-b}|u|^{\alpha+2}
     + \frac{4b}{\alpha+2}|x|^{-b-2}x_j a_j(x)|u|^{\alpha+2}
     \,dx.
     \label{eq:localized-virial-identity-sec4}
\end{align}
Here, as usual, repeated indices are summed.
\end{lemma}

\begin{proof}
This is the standard virial computation, adapted to the inhomogeneous potential
$|x|^{-b}$. One differentiates \eqref{eq:localized-virial-action-sec4} in time,
substitutes the equation $iu_t=-\Delta u+|x|^{-b}|u|^{\alpha}u$, and integrates by
parts. The terms coming from the linear Schr\"odinger flow yield the familiar
quadratic form
$4\,\mathrm{Re}(a_{jk}\partial_j u\overline{\partial_k u})-a_{jjkk}|u|^2$,
while the nonlinear part produces the two weighted terms in
\eqref{eq:localized-virial-identity-sec4}; the second of these is precisely the
correction coming from the spatial derivative of $|x|^{-b}$.
\end{proof}

We now choose the standard localized quadratic weight. Fix $R>1$ and let
$a_R:\R^d\to\R$ be radial such that
\begin{equation}\label{eq:weight-aR-sec4}
  a_R(x)=|x|^2 \quad \text{for } |x|\le R,
  \qquad
  a_R(x)=C_0R^2 \quad \text{for } |x|\ge 2R,
\end{equation}
for some fixed $C_0>1$, with the additional derivative bounds
\begin{equation}\label{eq:weight-aR-derivative-bounds}
  |\partial^{\gamma}a_R(x)|\lesssim R^{2-|\gamma|}
  \qquad \text{for } R<|x|<2R,
  \quad |\gamma|\le 4.
\end{equation}

\begin{proposition}[No global compact scenario]\label{prop:no-global-compact-solution-sec4}
There is no minimal kinetic-energy blowup solution of the form described in
the forward-time half of Theorem~\ref{prop:normalized-rigidity-scenarios}(ii).
\end{proposition}

\begin{proof}
Assume for contradiction that $u:[0,\infty)\times\R^d\to\CC$ is a minimal kinetic-energy blowup solution
with $\inf_{t\ge0}N(t)\ge1$. Fix $\eta>0$ and choose $R=R(\eta)$ as in
Lemma~\ref{lem:tightness-weighted-sec4}. Define the localized virial action
$M_R(t):=M_{a_R}(t)$.

\medskip
\noindent
\textsf{Step 1 (bound the virial action).}
Since $\nabla a_R$ is supported in $|x|\le 2R$ and
$|\nabla a_R(x)|\lesssim R$, we have
\begin{align*}
  |M_R(t)|
  &\lesssim R \int_{|x|\le 2R}|u(t,x)|\,|\nabla u(t,x)|\,dx \\
  &\le R\|u(t)\|_{L_x^2(|x|\le 2R)}\|\nabla u(t)\|_{L_x^2}.
\end{align*}
By H\"older inequality and Sobolev embedding,
\[
  \|u(t)\|_{L_x^2(|x|\le 2R)}
  \lesssim R\|u(t)\|_{L_x^{\frac{2d}{d-2}}}
  \lesssim R\|\nabla u(t)\|_{L_x^2}.
\]
Hence
\begin{equation}\label{eq:virial-action-bound-sec4}
  \sup_{t\ge0}|M_R(t)|\lesssim_u R^2.
\end{equation}

\medskip
\noindent
\textsf{  Step 2 (lower bound for the virial derivative).}
Applying Lemma~\ref{lem:localized-virial-identity-sec4} with $a=a_R$, and using
that $a_R(x)=|x|^2$ on $|x|\le R$, we obtain
\begin{equation}\label{eq:virial-derivative-main-simplified-sec4}
  \frac{d}{dt}M_R(t)
  = 8\int_{|x|\le R}
    \Bigl(|\nabla u(t,x)|^2 + |x|^{-b}|u(t,x)|^{\alpha+2}\Bigr)
    \,dx
    + \mathrm{Err}_R(t).
\end{equation}
The derivative bounds \eqref{eq:weight-aR-derivative-bounds} imply that the error
term is supported in the annulus $R<|x|<2R$ and satisfies
\begin{equation}\label{eq:virial-error-bound-sec4}
  |\mathrm{Err}_R(t)|
  \lesssim
  \int_{|x|>R}
  \Bigl(|\nabla u(t,x)|^2
  + |x|^{-b}|u(t,x)|^{\alpha+2}
  + |x|^{-2}|u(t,x)|^2\Bigr)
  \,dx.
\end{equation}
By Lemma~\ref{lem:tightness-weighted-sec4},
\begin{equation}\label{eq:virial-error-small-sec4}
  \sup_{t\ge0}|\mathrm{Err}_R(t)|\le C\eta.
\end{equation}
On the other hand,
\begin{align*}
  \int_{|x|\le R}
  \Bigl(|\nabla u|^2 + |x|^{-b}|u|^{\alpha+2}\Bigr)
  &= \int_{\R^d}
     \Bigl(|\nabla u|^2 + |x|^{-b}|u|^{\alpha+2}\Bigr)
     \,dx + O(\eta) \\
  &\ge 2E(u) + O(\eta),
\end{align*}
Choosing $\eta$ sufficiently small,
we conclude from \eqref{eq:virial-derivative-main-simplified-sec4} that there
exists $c_0=c_0(u)>0$ such that
\begin{equation}\label{eq:virial-derivative-positive-sec4}
  \frac{d}{dt}M_R(t)\ge c_0
  \qquad \text{for all } t\ge0.
\end{equation}

\medskip
\noindent
\textsf{ Step 3 (Contradiction).}
Integrating \eqref{eq:virial-derivative-positive-sec4} on $[0,T]$ and using
\eqref{eq:virial-action-bound-sec4}, we obtain
\[
  c_0T
  \le M_R(T)-M_R(0)
  \le 2\sup_{t\ge0}|M_R(t)|
  \lesssim_u R^2.
\]
Since $R$ is fixed, this is impossible for sufficiently large $T$. The
contradiction proves that the global compact scenario cannot occur.
\end{proof}

\subsection{Completion of the proof of Theorem~\ref{thm:main}}

Assume for contradiction that Theorem~\ref{thm:main} fails. By
Theorem~\ref{prop:critical-element}, there exists a minimal kinetic-energy blowup solution. By
Theorem~\ref{prop:normalized-rigidity-scenarios}, after replacing it by another
minimal kinetic-energy blowup solution if necessary, we are in either the finite-time scenario or the
global compact scenario.

The finite-time scenario is impossible by
Proposition~\ref{prop:no-finite-time-blowup-sec4}. The global compact scenario
is impossible by Proposition~\ref{prop:no-global-compact-solution-sec4}. Hence no
minimal kinetic-energy blowup solution exists, contradicting Theorem~\ref{prop:critical-element}.
Therefore Theorem~\ref{thm:main} follows.

\appendix

\section{Some standard facts on Sobolev-Lorentz spaces}\label{sec:appendix-lorentz}

\begin{lemma}[Properties of Lorentz spaces \cite{Neil-1963}]
\label{lem:pro-lorentz}
\mbox{}
\begin{itemize}
\item For $1<p<\infty$,  $L^{p,p}(\mathbb{R}^d)$ coincides with the standard Lebesgue space $L^{p}(\mathbb{R}^d)$, while $L^{p,\infty}(\mathbb{R}^d)$ coincides with the weak $L^{p}(\mathbb{R}^d)$ space.
\item For $1<p<\infty$ and $0<q_1<q_2\le\infty$, $L^{p,q_1}(\mathbb{R}^d)\subset L^{p,q_2}(\mathbb{R}^d)$.
\item For $1<p<\infty$, $0<q\le\infty$, and $\theta>0$, $\bigl\||u|^\theta\bigl\|_{L^{p,q}}=\|u\|_{L^{\theta p,\theta q}}^\theta$.
\item For $b>0$, $|x|^{-b}\in L^{\frac{d}{b},\infty}(\mathbb{R}^d)$ and
$
\bigl\||x|^{-b}\bigr\|_{L^{\frac{d}{b},\infty}}=|B(0,1)|^{\frac{d}{b}},
$
where $B(0,1)$ is the unit ball of $\mathbb{R}^d$.
\end{itemize}
\end{lemma}

\begin{lemma}[H\"older inequality in Lorentz spaces {\cite{Neil-1963}}]
\label{lem:lorentz-holder}
Let \(1<p_i<\infty\) and \(1\le q_i\le\infty\). Assume
\[
\frac1{p_0}=\frac1{p_1}+\frac1{p_2},
\qquad
\frac1{q_0}\le \frac1{q_1}+\frac1{q_2}.
\]
Then
\[
\|uv\|_{L^{p_0,q_0}(\R^d)}
\lesssim
\|u\|_{L^{p_1,q_1}(\R^d)}
\|v\|_{L^{p_2,q_2}(\R^d)}.
\]
\end{lemma}

\begin{lemma}[Bernstein inequality in Lorentz space \cite{Liu-Miao-Zheng2025}]
\label{lem:bernstein-lorentz}
Let $N\in 2^{\Z}$, $s\ge 0$, $1<p_1\le p_2<\infty$, and $1\le q_1\le q_2\le \infty$.
Then the following estimates hold:
\begin{align}
  \|P_N u\|_{L^{p_2,q_2}(\R^d)}
  &\lesssim N^{d\left(\frac1{p_1}-\frac1{p_2}\right)}
  \|u\|_{L^{p_1,q_1}(\R^d)},
  \label{eq:bernstein-dyadic-lorentz}\\
  \bigl\||\nabla|^{s}P_{\le N}u\bigr\|_{L^{p_2,q_2}(\R^d)}
  &\lesssim N^{s+d\left(\frac1{p_1}-\frac1{p_2}\right)}
  \|u\|_{L^{p_1,q_1}(\R^d)},
  \label{eq:bernstein-low-lorentz}\\
  \|P_{>N}u\|_{L^{p_2,q_2}(\R^d)}
  &\lesssim N^{-s+d\left(\frac1{p_1}-\frac1{p_2}\right)}
  \bigl\||\nabla|^{s}u\bigl\|_{L^{p_1,q_1}(\R^d)}.
  \label{eq:bernstein-high-lorentz}
\end{align}
\end{lemma}

\begin{lemma}[Sobolev embedding in Lorentz spaces]\label{lem:sobolev-lorentz}
Let $0<s<d$, $1<p<\frac ds$, and define $r$ by
\[
\frac1r=\frac1p-\frac{s}{d}.
\]
Then for any $1\le q\le \infty$
\[
\|u\|_{L^{r,q}(\mathbb{R}^d)}\lesssim \bigl\||\nabla|^s u \bigl\|_{L^{p,q}(\mathbb{R}^d)}.
\]
In particular,
\[
\dot H^1(\mathbb{R}^d)=\dot W^{1;2,2}(\mathbb{R}^d)\hookrightarrow L^{\frac{2d}{d-2},2}(\mathbb{R}^d).
\]
\end{lemma}

\begin{proof}
This follows from the Riesz-potential representation $|\nabla|^{-s}u=c_{d,s}|x|^{s-d}*u$, together with Young's inequality and the fact that $|x|^{s-d}\in L^{\frac{d}{d-s},\infty}(\mathbb{R}^d)$.
\end{proof}

\begin{lemma}[Gagliardo--Nirenberg inequality in Lorentz space \cite{Wei-wang-2023}]\label{lem:G-N-lorentz}
Let $1<q_i, r_i\le \infty$ for $i=0, 1$ and $0<s<\sigma<\infty$. Assume
\[
\frac{1}{p_i}
=
\left(1-\frac{s}{\sigma}\right)\frac{1}{q_i}
+
\frac{s}{\sigma}\frac1{r_i}
\]
for $i=0,1$. Then
\[
\bigl\||\nabla|^s u\bigr\|_{L^{p_0,p_1}(\mathbb{R}^d)}
\lesssim
\|u\|_{L^{q_0,q_1}(\mathbb{R}^d)}^{1-\frac{s}{\sigma}}
\||\nabla|^\sigma u\|_{L^{r_0,r_1}(\mathbb{R}^d)}^{\frac{s}{\sigma}}.
\]
\end{lemma}

\begin{lemma}[Hardy inequality in Lorentz space]\label{lem:hardy}
Let $1<p<d$, we have
\[
\bigl\||x|^{-1}u\bigr\|_{L^{p,2}(\mathbb{R}^d)}
\lesssim
\|\nabla u\|_{L^{p,2}(\mathbb{R}^d)}.
\]
More generally, for any $0<s<\frac dp$
\[
\bigl\||x|^{-s}u\bigr\|_{L^{p,2}(\mathbb{R}^d)}
\lesssim
\||\nabla|^s u\|_{L^{p,2}(\mathbb{R}^d)}.
\]
\end{lemma}

\begin{proof}
This is the standard Hardy inequality together with interpolation and Sobolev embedding.
\end{proof}

\begin{lemma}[Fractional Leibniz rule \cite{Cruz-Naibo2016}]\label{lem:fractional-leibniz}
Let $1<p_i<\infty$ and $1\le q_i\le\infty$ for $i=0,1,2,3,4$. Assume
\[
\frac1{p_0}=\frac1{p_1}+\frac1{p_2}=\frac1{p_3}+\frac1{p_4},
\qquad
\frac1{q_0}\le \frac1{q_1}+\frac1{q_2},
\qquad
\frac1{q_0}\le \frac1{q_3}+\frac1{q_4}.
\]
Then
\[
\bigl\||\nabla|^{s}(uv)\bigr\|_{L^{p_0,q_0}(\mathbb{R}^d)}
\lesssim
\bigl\||\nabla|^{s}u\bigr\|_{L^{p_1,q_1}(\mathbb{R}^d)}\|v\|_{L^{p_2,q_2}(\mathbb{R}^d)}
+
\|u\|_{L^{p_3,q_3}(\mathbb{R}^d)} \bigl\| |\nabla|^{s}v \bigr\|_{L^{p_4,q_4}(\mathbb{R}^d)}.
\]
\end{lemma}

\begin{lemma}[Fractional chain rule \cite{Aloui-Tayachi-2021}]\label{lem:fractional-chain}
Let $0\le s \le1$, $F\in C^1(\mathbb{C},\mathbb{C})$ and $1<p_i<\infty$, $1\le q_i<\infty$ for $i=0,1,2$. Assume
\[
  \frac1{p_0}=\frac1{p_1}+\frac1{p_2},
  \qquad
  \frac1{q_0}= \frac1{q_1}+\frac1{q_2}.
\]
Then
\[
\bigl\||\nabla|^{s}F(u)\bigr\|_{L^{p_0,q_0}(\mathbb{R}^d)}
\lesssim
\|F'(u)\|_{L^{p_1,q_1}(\mathbb{R}^d)} \bigl\||\nabla|^{s}u\bigr\|_{L^{p_2,q_2}(\mathbb{R}^d)}.
\]
\end{lemma}

\begin{lemma}[Fractional chain rule for H\"older continuous functions \cite{Killip-Visan-2013}] \label{lem:frac-chain-holder}
Let $0< s <\alpha<1$, $F\in C^{0,\alpha}(\mathbb{C},\mathbb{C})$ and $1<p_i<\infty$, $1\le q_i<\infty$ for $i=0,1,2$. Assume
\[
\frac{1}{p_0}=\frac{\alpha-\frac{s}{\sigma}}{p_1}+\frac{\frac{s}{\sigma}}{p_2},
 \qquad
\frac{1}{q_0}\le \frac{\alpha-\frac{s}{\sigma}}{q_1}+\frac{\frac{s}{\sigma}}{q_2}
\]
Then for any $\frac{s}{\alpha}<\sigma<1$, we have
\begin{equation}
\bigl\||\nabla|^s G(u)\bigr\|_{L^{p_0,q_0}}
\lesssim
\|u\|^{\alpha-\frac{s}{\sigma}}_{L^{p_1,q_1}}
\bigl\||\nabla|^\sigma u\bigl\|^{\frac{s}{\sigma}}_{L^{p_2,q_2}}.
\end{equation}
\end{lemma}

\begin{lemma}[Weighted fractional multiplier]
\label{lem:weighted-fractional-multiplier}
Let \(0< s\le 1\), \(0<b<2\), and
$
1<p<\frac{d}{b+s}, 1\le q<\infty.
$
Set
\[
\frac1r=\frac1p-\frac{b}{d}.
\]
Then, for every \(F \in \dot W^{s;r,q}\),
\begin{equation}\label{eq:weighted-fractional-multiplier}
\bigl\||\nabla|^s(|x|^{-b}F)\bigr\|_{L^{p,q}(\R^d)}
\lesssim
\bigl\||\nabla|^sF\bigr\|_{L^{r,q}(\R^d)}.
\end{equation}
\end{lemma}

\begin{proof}
Set
\[
  \frac1m:=\frac1p-\frac{b+s}{d}=\frac1r-\frac{s}{d}.
\]
Since $1<p<\frac{d}{b+s}$, we have $1<r<\frac{d}{s}$ and $1<m<\infty$.
Applying Lemma~\ref{lem:fractional-leibniz}, we obtain
\begin{align*}
  \bigl\||\nabla|^s(|x|^{-b}F)\bigr\|_{L^{p,q}}
  &\lesssim
  \bigl\||\nabla|^s(|x|^{-b})\bigr\|_{L^{\frac{d}{b+s},\infty}}
  \|F\|_{L^{m,q}}
  +
  \bigl\||x|^{-b}\bigl\|_{L^{\frac{d}{b},\infty}}
  \bigl\||\nabla|^sF\bigr\|_{L^{r,q}}.
\end{align*}
Now
\[
  |x|^{-b}\in L^{\frac{d}{b},\infty}(\R^d),
  \qquad
  |\nabla|^s(|x|^{-b})\lesssim |x|^{-b-s}\in L^{\frac{d}{b+s},\infty}(\R^d).
\]
Moreover, since $\frac1m=\frac1r-\frac{s}{d}$, Lemma~\ref{lem:sobolev-lorentz} implies
\[
  \|F\|_{L^{m,q}(\R^d)}
  \lesssim
  \bigl\||\nabla|^sF\bigr\|_{L^{r,q}(\R^d)}.
\]
Substituting this into the previous estimate yields
\eqref{eq:weighted-fractional-multiplier}.
\end{proof}

\begin{lemma}[Dispersive estimates]\label{prop:dispersive-estimates}
	Let $2<p<\infty$ and $1\le q \le\infty$. Then for any $u\in L_x^{p',q}(\mathbb{R}^d)$
	\[
	\|e^{it\Delta}u_0\|_{L_x^{p,q}(\mathbb{R}^d)}
	\lesssim
	|t|^{-d\left(\frac12-\frac1p\right)}
	\|u_0\|_{L_x^{p',q}(\mathbb{R}^d)}
	\qquad \text{for } t\ne 0.
	\]
\end{lemma}

\begin{lemma}[Strichartz estimates]
	\mbox{}
	\begin{itemize}
		\item Let $(q,r)\in \Lambda$ with $r<\infty$. Then for any $u\in L^2(\mathbb{R}^d)$
		\[
		\left\|e^{it\Delta}u\right\|_{L_t^qL_x^{r,2}(\mathbb{R}\times\mathbb{R}^d)}
		\lesssim
		\|u\|_{L_x^2(\mathbb{R}^d)}.
		\]
		
		\item Let $(q_1,r_1),(q_2,r_2)\in \Lambda$ with $r_1,r_2<\infty$, $t_0\in\mathbb{R}$ and $I\subset\mathbb{R}$ be an interval containing $t_0$. Then for any $f\in L_t^{q_2'}L_x^{r_2',2}(I\times\mathbb{R}^d)$
		\[
		\left\|\int_{t_0}^{t} e^{i(t-\tau)\Delta}f(\tau)\,d\tau\right\|_{L_t^{q_1}L_x^{r_1,2}(I\times\mathbb{R}^d)}
		\lesssim
		\|f\|_{L_t^{q_2'}L_x^{r_2',2}(I\times\mathbb{R}^d)}.
		\]
	\end{itemize}
\end{lemma}

\section*{Acknowledgements}

This work was supported by the National Key Research and Development Program of China (No. 2023YFC2206100), and the National Natural Science Foundation of China (No. 12231008).


\begin{thebibliography}{99}


\bibitem{Aloui-Tayachi-2024}
L. Aloui and S. Tayachi,
\emph{Global existence and scattering for the inhomogeneous nonlinear Schr\"odinger equation},
J. Evol. Equ. \textbf{24} (2024), 61.
% unsuc

\bibitem{Aloui-Tayachi-2021}
L. Aloui and S. Tayachi,
\emph{Local well-posedness for the inhomogeneous nonlinear Schr\"odinger equation},
Discrete Contin. Dyn. Syst. \textbf{41} (2021), 5409--5437.

\bibitem{Bourgain-1999}
J. Bourgain,
\emph{Global well-posedness of defocusing critical nonlinear Schr\"odinger equation in the radial case},
J. Amer. Math. Soc. \textbf{12} (1999), 145--171.

% Online status checked: preprint only
\bibitem{Campos-Farah-Murphy-2026}
L. Campos, L. G. Farah, and J. Murphy,
\emph{Threshold solutions for the $3d$ cubic INLS: the energy-critical case},
arXiv:2601.05349.

\bibitem{Campos-Correia-Farah-2025}
L. Campos, S. Correia, and L. G. Farah,
\emph{Sharp well-posedness and ill-posedness results for the inhomogeneous NLS equation},
Nonlinear Anal. Real World Appl. \textbf{85} (2025), 104336.

\bibitem{Cardoso-Campos-2022}
M. Cardoso, L. G. Farah, C. M. Guzm\'an, and J. Murphy,
\emph{Scattering below the ground state for the intercritical non-radial inhomogeneous NLS},
Nonlinear Anal. Real World Appl. \textbf{68} (2022), 103687.

\bibitem{Cazenave-Weissler-1988}
T. Cazenave and F. B. Weissler,
\emph{The Cauchy problem for the nonlinear Schr\"odinger equation in $H^1$},
Manuscripta Math. \textbf{61} (1988), 477--494.

\bibitem{Cazenave-2003}
T. Cazenave,
\emph{Semilinear Schr\"odinger Equations},
Courant Lecture Notes in Mathematics, Vol.~10, New York Univ., Courant Inst. Math. Sci., New York; Amer. Math. Soc., Providence, RI, 2003.

\bibitem{Cho-Lee-2021}
Y. Cho and K. Lee,
\emph{On the focusing energy-critical inhomogeneous NLS: weighted space approach},
Nonlinear Anal. \textbf{205} (2021), 112261.

\bibitem{Cho-Hong-Lee-2020}
Y. Cho, S. Hong, and K. Lee,
\emph{On the global well-posedness of focusing energy-critical inhomogeneous NLS},
J. Evol. Equ. \textbf{20} (2020), 1349--1380.

\bibitem{CKSTT-2008}
J. Colliander, M. Keel, G. Staffilani, H. Takaoka, and T. Tao,
\emph{Global well-posedness and scattering for the energy-critical nonlinear Schr\"odinger equation in $\R^3$},
Ann. of Math. (2) \textbf{167} (2008), 767--865.

\bibitem{Cruz-Naibo2016}
D. Cruz-Uribe and V. Naibo,
\emph{Kato--Ponce inequalities on weighted and variable Lebesgue spaces},
Differential Integral Equations \textbf{29} (2016), 801--836.

\bibitem{Dao-2018}
N. A. Dao, J. I. D\'{\i}az, and Q. H. Nguyen,
\emph{Generalized Gagliardo--Nirenberg inequalities using Lorentz spaces, BMO, H\"older spaces and fractional Sobolev spaces},
Nonlinear Anal. \textbf{173} (2018), 146--153.

\bibitem{Dinh-2019}
V. D. Dinh,
\emph{Energy scattering for a class of the defocusing inhomogeneous nonlinear Schr\"odinger equation},
J. Evol. Equ. \textbf{19} (2019), 411--434.

\bibitem{Dodson-2019}
B. Dodson,
\emph{Global well-posedness and scattering for the focusing, cubic Schr\"odinger equation in dimension $d=4$},
Ann. Sci. \'Ec. Norm. Sup\'er. (4) \textbf{52} (2019), 139--180.

\bibitem{Farah-Guzman-2017}
L. G. Farah and C. M. Guzm\'an,
\emph{Scattering for the radial $3D$ cubic focusing inhomogeneous nonlinear Schr\"odinger equation},
J. Differential Equations \textbf{262} (2017), 4175--4231.

\bibitem{Farah-2016}
L. G. Farah,
\emph{Global well-posedness and blow-up on the energy space for the inhomogeneous nonlinear Schr\"odinger equation},
J. Evol. Equ. \textbf{16} (2016), 193--208.

\bibitem{Genoud-Stuart-2008}
F. Genoud and C. A. Stuart,
\emph{Schr\"odinger equations with a spatially decaying nonlinearity: existence and stability of standing waves},
Discrete Contin. Dyn. Syst. \textbf{21} (2008), 137--186.

\bibitem{Gill-2000}
T. S. Gill,
\emph{Optical guiding of laser beam in nonuniform plasma},
Pramana J. Phys. \textbf{55} (2000), 835--842.
% unsuc

\bibitem{Grillakis-2000}
M. G. Grillakis,
\emph{On nonlinear Schr\"odinger equations},
Comm. Partial Differential Equations \textbf{25} (2000), 1827--1844.

\bibitem{Guzman-Xu-2025}
C. M. Guzm\'{a}n and C. Xu,
\emph{Dynamics of the non-radial energy-critical inhomogeneous {NLS}},
Potential Anal. \textbf{63} (2025), 601--630.

\bibitem{Guzman-Murphy-2021}
C. M. Guzm\'an and J. Murphy,
\emph{Scattering for the non-radial energy-critical inhomogeneous NLS},
J. Differential Equations \textbf{295} (2021), 187--210.

\bibitem{Guzman-2017}
C. M. Guzm\'an,
\emph{On well posedness for the inhomogeneous nonlinear Schr\"odinger equation},
Nonlinear Anal. Real World Appl. \textbf{37} (2017), 249--286.

\bibitem{Guzman-Keraani-Xu-2025}
C. M. Guzm\'an, S. Keraani, and C. Xu,
\emph{Global dynamics of the non-radial energy-critical inhomogeneous biharmonic NLS},
arXiv:2508.02796.

\bibitem{Hasegawa-Tappert-1973}
A. Hasegawa and F. Tappert,
\emph{Transmission of stationary nonlinear optical pulses in dispersive dielectric fibers. I. Anomalous dispersion},
Appl. Phys. Lett. \textbf{23} (1973), 142--144.

\bibitem{Kenig-Merle-2006}
C. E. Kenig and F. Merle,
\emph{Global well-posedness, scattering and blow-up for the energy-critical, focusing, non-linear Schr\"odinger equation in the radial case},
Invent. Math. \textbf{166} (2006), 645--675.

\bibitem{Keraani-2001}
S. Keraani,
\emph{On the defect of compactness for the Strichartz estimates of the Schr\"odinger equations},
J. Differential Equations \textbf{175} (2001), 353--392.

\bibitem{Killip-Visan-2013}
R. Killip and M. Visan,
\emph{Nonlinear Schr\"odinger equations at critical regularity},
in \emph{Evolution Equations}, Clay Math. Proc., Vol.~17, Amer. Math. Soc., Providence, RI, 2013, pp.~325--437.

\bibitem{Killip-Visan-2010}
R. Killip and M. Visan,
\emph{The focusing energy-critical nonlinear Schr\"odinger equation in dimensions five and higher},
Amer. J. Math. \textbf{132} (2010), 361--424.

\bibitem{Kim-Lee-Seo-2021}
J. Kim, Y. Lee, and I. Seo,
\emph{On well-posedness for the inhomogeneous nonlinear Schr\"odinger equation in the critical case},
J. Differential Equations \textbf{280} (2021), 179--202.

\bibitem{Lee-2025-strong}
Y. Lee,
\emph{The $3D$ energy-critical inhomogeneous nonlinear Schr\"odinger equation with strong singularity},
arXiv:2501.02697.

\bibitem{Liu-Tripathi-1994}
C. S. Liu and V. K. Tripathi,
\emph{Laser guiding in an axially nonuniform plasma channel},
Phys. Plasmas \textbf{1} (1994), 3100--3103.
% unsuc

\bibitem{Liu-Xu-2025-super}
X. Liu and C. Xu,
\emph{The defocusing energy-supercritical inhomogeneous NLS in four space dimension},
J. Math. Anal. Appl. \textbf{554} (2026), 129968.

\bibitem{Liu-Miao-Zheng2025}
X. Liu, C. Miao, and J. Zheng,
\emph{Global well-posedness and scattering for mass-critical inhomogeneous {NLS} when $d\ge 3$},
Math. Z. \textbf{311} (2025), 54.

\bibitem{Liu-Yang-Zhang-2024}
X. Liu, K. Yang, and T. Zhang,
\emph{Dynamics of threshold solutions for the energy-critical inhomogeneous NLS},
arXiv:2409.00073.

\bibitem{Miao-Murphy-Zheng-2021}
C. Miao, J. Murphy, and J. Zheng,
\emph{Scattering for the non-radial inhomogeneous {NLS}},
Math. Res. Lett. \textbf{28} (2021), 1481--1504.

\bibitem{Neil-1963}
R. O'Neil,
\emph{Convolution operators and $L(p,q)$ spaces},
Duke Math. J. \textbf{30} (1963), 129--142.

\bibitem{Park-Defocusing-2024}
D. Park,
\emph{Global well-posedness and scattering of the defocusing energy-critical inhomogeneous nonlinear Schr\"odinger equation with radial data},
J. Math. Anal. Appl. \textbf{536} (2024), 128202.

\bibitem{Park-Focusing-2024}
D. Park,
\emph{Global well-posedness and scattering for the focusing energy-critical inhomogeneous nonlinear Schr\"odinger equation with non-radial data},
arXiv:2410.12321.

\bibitem{Ryckman-Visan-2007}
E. Ryckman and M. Visan,
\emph{Global well-posedness and scattering for the defocusing energy-critical nonlinear Schr\"odinger equation in $\R^{1+4}$},
Amer. J. Math. \textbf{129} (2007), 1--60.

\bibitem{Tao-Visan-2005}
T. Tao and M. Visan,
\emph{Stability of energy-critical nonlinear Schr\"odinger equations in high dimensions},
Electron. J. Differential Equations \textbf{2005} (2005), No. 118, 1--28.

\bibitem{Tao-2005}
T. Tao,
\emph{Global well-posedness and scattering for the higher-dimensional energy-critical nonlinear Schr\"odinger equation for radial data},
New York J. Math. \textbf{11} (2005), 57--80.

\bibitem{Visan-2007}
M. Visan,
\emph{The defocusing energy-critical nonlinear Schr\"odinger equation in higher dimensions},
Duke Math. J. \textbf{138} (2007), 281--374.

\bibitem{Wei-wang-2023}
W. Wei, Y. Wang, and Y. Ye,
\emph{Gagliardo--Nirenberg inequalities in Lorentz type spaces},
J. Fourier Anal. Appl. \textbf{29} (2023), 35.

\bibitem{Zakharov-Shabat-1972}
V. E. Zakharov and A. B. Shabat,
\emph{Exact theory of two-dimensional self-focusing and one-dimensional self-modulation of waves in nonlinear media},
Sov. Phys. JETP \textbf{34} (1972), 62--69.

\end{thebibliography}
\end{document}